\documentclass[12pt]{article}
\usepackage[utf8]{inputenc}

\usepackage[left=1in,right=1in,top=1in,bottom=1in]{geometry} 
\usepackage{parskip} 
\usepackage{enumitem} 
\usepackage{graphicx, wrapfig, float, subcaption} 
\usepackage{bookmark} 
\usepackage{comment}
\usepackage{xcolor}
\usepackage{titlesec}

\usepackage{etoolbox}
\setlist[itemize]{itemsep=0pt}
\setlist[enumerate]{itemsep=0pt}
\setlist[enumerate]{label=(\roman*)}
\hypersetup{colorlinks,urlcolor=blue}

\renewcommand{\b}{\textbf}
\renewcommand{\i}{\textit}
\renewcommand{\t}{\text}

\usepackage{amsmath,amsfonts,amssymb,amsthm}
\usepackage[most]{tcolorbox} 

\newtheorem{thm}{Theorem}[section]
\newtheorem{lem}[thm]{Lemma}
\newtheorem{cor}[thm]{Corollary}
\newtheorem{defn}[thm]{Definition}
\newtheorem{prop}[thm]{Proposition}
\newtheorem{conj}[thm]{Conjecture}

\newtheorem{stat}{}[thm]

\newenvironment{restate}[1]
    {\def\thethm{#1}\begin{thm}}
    {\end{thm}\addtocounter{thm}{-1}}
    
\makeatletter
\def\moverlay{\mathpalette\mov@rlay}
\def\mov@rlay#1#2{\leavevmode\vtop{%
   \baselineskip\z@skip \lineskiplimit-\maxdimen
   \ialign{\hfil$\m@th#1##$\hfil\cr#2\crcr}}}
\newcommand{\charfusion}[3][\mathord]{
    #1{\ifx#1\mathop\vphantom{#2}\fi
        \mathpalette\mov@rlay{#2\cr#3}
      }
    \ifx#1\mathop\expandafter\displaylimits\fi}
\makeatother

\newcommand{\m}{\mathbb}
\renewcommand{\c}{\mathcal}

\usepackage{tikz}
\usetikzlibrary{calc}

\def\blueedge{\pgfkeysalso{blue}}
\def\cyanedge{\pgfkeysalso{cyan}}
\def\greenedge{\pgfkeysalso{green}}
\def\yellowedge{\pgfkeysalso{yellow!80!black}}

\newenvironment{drawing}[1]{\begin{tikzpicture}[thick, scale=#1,nodelabel/.style={rounded corners,fill=none,inner sep=5pt,draw=none},
every edge/.append code = {%
    \global\let\currenttarget\tikztotarget 
    \pgfkeysalso{append after command={(\currenttarget)}}
    \edgecolor
}
]\let\edgecolor\rededge}{\end{tikzpicture}}

\tikzstyle{bnode}=[circle, draw, fill=black, inner sep=0pt, minimum width=5pt]
\tikzstyle{wnode}=[circle, draw, fill=white, inner sep=0pt, minimum width=5pt]

\newcommand\ps{-8pt}

\newcommand{\mcg}[0]{matching covered graph}
\newcommand{\bmcg}[0]{bipartite matching covered graph}
\newcommand{\mbmcg}[0]{minimal bipartite matching covered graph}
\newcommand{\embmcg}[0]{extremal minimal bipartite matching covered graph}
\newcommand{\ex}[1]{\mbox{$#1$-extendable}}
\newcommand{\kex}[1]{\mbox{$#1$-extendable} graph}
\newcommand{\bkex}[1]{\mbox{$#1$-extendable} bipartite graph}
\newcommand{\mbkex}[1]{minimal \mbox{$#1$-extendable} bipartite graph}

\newcommand{\mtc}[0]{minimal \mbox{\kC{2}} graph}
\newcommand{\emtc}[0]{extremal minimal \mbox{\kC{2}} graph}
\newcommand{\cmcs}[0]{conformal matching covered subgraph}
\newcommand{\mc}[0]{matching covered}
\newcommand{\pema}[0]{perfect matching}
\newcommand{\kC}[1]{\mbox{$#1$-connected}}

\title{\vspace{-2.25cm}\Huge Extremal minimal \\bipartite matching covered graphs \footnote{Supported by IC\&SR IIT Madras}}
\author{Amit Kumar Mallik \\
\small Indian Institute of Technology Bombay\\[-0.8ex]
\small\tt amit.km287@gmail.com\\
\and
Ajit A. Diwan\\
\small Indian Institute of Technology Bombay\\[-0.8ex]
\small\tt aad@cse.iitb.ac.in\\
\and
Nishad Kothari\\
\small Indian Institute of Technology Madras\\[-0.8ex]
\small\tt nishadkothari@gmail.com\\
}

\begin{document}

\newcommand{\eem}[0]{$2$-edge extremal}
\newcommand{\een}[0]{$2$-edge $n$-extremal}
\newcommand{\evm}[0]{$2$-vertex extremal}
\newcommand{\evn}[0]{$2$-vertex $n$-extremal}
\newcommand{\ee}[0]{edge extremal}
\newcommand{\htree}[0]{Halin tree}
\newcommand{\ktree}[1]{$#1$-tree}
\newcommand{\kregular}[1]{$#1$-regular}
\newcommand{\isojoin}[0]{isomorphic leaf matching}
\newcommand{\kisojoin}[1]{isomorphic \mbox{$#1$-leaf} matching}
\newcommand{\retract}[0]{partial retract}
\newcommand{\bicontract}[0]{restricted bicontraction}
\newcommand{\bisplit}[0]{restricted bisplitting}
\newcommand{\tedge}[0]{\mbox{$3$-edge}}
\newcommand{\balanced}[1]{special}
\newcommand{\removable}[0]{superfluous}

\maketitle

\centerline{The first and last authors dedicate this work to their coauthor,}\centerline{the late Ajit A. Diwan}

\begin{abstract}

A connected graph, on four or more vertices, is \i{\mc{}} (aka \i{\ex{1}}) if every edge is present in some \pema{}. An ear decomposition theorem (similar to the one for \kC{2} graphs) exists for \bmcg{}s due to Hetyei. From the results and proofs of Lov\'asz and Plummer [{\em Matching Theory}, Annals of Discrete Math. 29, 1986], that rely on Hetyei's Theorem, one may deduce that any \mbmcg{} has at least $2(m-n+2)$ vertices of degree two (where \i{minimal} means that deleting any edge results in a graph that is not \mc{}); such a graph is \emph{extremal} if it attains the stated bound. 

In this paper, we provide a complete characterization of the class of \embmcg{}s. In particular, we prove that every such graph~$G$ is obtained from two copies of a tree devoid of degree two vertices, say $T$ and $T'$, by adding edges --- each of which joins a leaf of $T$ with the corresponding leaf of $T'$.

Apart from the aforementioned bound, there are four other bounds that appear in, or may be deduced from, the work of Lov\'asz and Plummer. Each of these bounds leads to a notion of extremality. In this paper, we obtain a complete characterization of all of these extremal classes and also establish relationships between them. Two of our characterizations are in the same spirit as the one stated above. For the remaining two extremal classes, we reduce each of them to one of the already characterized extremal classes using standard matching theoretic operations.  

A connected graph is \i{\ex{k}} if it has a matching of cardinality $k$ and each such matching extends to a perfect matching. We also discuss bounds proved by Lou [{\em On the structure of minimally \ex{n} bipartite graphs}, Discrete Math. 202 (1), 1999] for \mbkex{k}s (where \i{minimal} means that deleting any edge results in a graph that is not \ex{k}). We conjecture stronger bounds and provide evidence for our conjectures by constructing tight examples that are straightforward generalizations of the ones that appear
in the \ex{1} case.




\end{abstract}


\section{Introduction and summary}
All graphs considered here are loopless; however, we allow parallel/multiple edges. For notation and terminology, we largely follow Bondy and Murty \cite{bomu08}. For a graph $G:=(V,E)$, its \i{order} is the number of vertices denoted by $n$, and its \i{size} is the number of edges denoted by $m$. We use the notation $G[A,B]$ to denote a bipartite graph with specified color classes $A$ and $B$. 


A graph is \i{matchable} if it has a \pema{} and an edge is \i{matchable} if it belongs to some \pema{}. A subgraph $H$ of a graph $G$ is \i{conformal} if $G-H$ is matchable. A connected graph of order four or more is \i{matching covered} if every edge is matchable; these graphs are also referred to as \i{\ex{1}} in the literature since each edge extends to a \pema{}; see Lov\'asz and Plummer \cite{lopl86}. There is an elegant ear decomposition theory for the class of \bmcg{}s, due to Hetyei \cite{het64}, that may be viewed as a refinement of the more well-known ear decomposition theory for the larger class of \kC{2} graphs due to Whitney \cite{whit33}. We describe this below.

An \i{ear of $H$ in $G$} is an odd path of $G$ whose ends are in $H$ but is otherwise disjoint from~$H$. For a bipartite graph $G$, a sequence of subgraphs $(G_0,G_1,\dots ,G_r)$ is an \i{ear decomposition of $G$} if: (i) $C:=G_0$ is an (even) cycle, (ii) $G_{i+1}=G_i\cup P_{i+1}$ where $P_{i+1}$ is an ear of $G_i$ for each $i\in\{0,1,\dots, r-1\}$, and (iii) $G_r=G$. It is easy to observe that each subgraph $G_i$ is a \cmcs{} of $G$, and that $r=m-n$. We refer to $(C,P_1,\dots ,P_r)$ as the \i{associated ear sequence}, or simply an \i{ear sequence of $G$}. Figure~\ref{fig: ex ear decomp} shows a \bmcg{} and an ear sequence for the same.

\begin{figure}[htb]
    \centering
    \begin{subfigure}{0.49\linewidth}
        \centering
        \begin{drawing}{1}
            \draw(-1,1)--(2,1)--(1,0)--(0,0)--(-1,1);
            \draw(-1,1)--(-2,0)--(-1,-1)--(0,0);
            \draw(2,1)--(3,1)--(3,0)--(1,0);
            \draw(-1,-1)--(2,-1)--(3,-1)--(3,0);
            \draw(1,0)--(2,-1);
            
            \draw (0,0)node[wnode]{}(-1,1)node[bnode]{}(-2,0)node[wnode]{}(-1,-1)node[bnode]{};
            \draw (1,0)node[bnode]{}(2,1)node[wnode]{}(3,1)node[bnode]{}(3,0)node[wnode]{}(3,-1)node[bnode]{}(2,-1)node[wnode]{};
        \end{drawing}
    \end{subfigure}
    \begin{subfigure}{0.49\linewidth}
        \centering
        \begin{drawing}{1}
            
            \draw[color=red] (-1,1)--(2,1)--(1,0)--(0,0)--(-1,1);
            \draw[color=red] (0.5,0.5)node[nodelabel]{$C$};

            \draw[color=blue] (-1,1)--(-2,0)--(-1,-1)--(0,0);
            \draw[color=blue] (-1.5,-0.5)node[above right,nodelabel]{$P_1$};
            
            \draw[color=green!70!black] (2,1)--(3,1)--(3,0)--(1,0);
            \draw[color=green!70!black] (3,0.5)node[left,nodelabel]{$P_2$};
            
            \draw[color=orange!80!black] (-1,-1)--(2,-1)--(3,-1)--(3,0);
            \draw[color=orange!80!black] (0.5,-1)node[above left,nodelabel]{$P_3$};
            
            \draw[color=cyan] (1,0)--(2,-1);  
            \draw[color=cyan] (1.5,-0.5)node[right,nodelabel]{$P_4$};  
        
            \draw (0,0)node[wnode]{}(-1,1)node[bnode]{}(-2,0)node[wnode]{}(-1,-1)node[bnode]{};
            \draw (1,0)node[bnode]{}(2,1)node[wnode]{}(3,1)node[bnode]{}(3,0)node[wnode]{}(3,-1)node[bnode]{}(2,-1)node[wnode]{};
        \end{drawing}
    \end{subfigure}
    
    \caption{a \bmcg{} and its ear decomposition}
    \label{fig: ex ear decomp}
\end{figure}

The following is the aforementioned result by Hetyei that we will find useful in order to establish that certain bipartite graphs are indeed \mc{}.

\begin{thm}{\sc[Ear Decomposition Theorem]}\\
    A bipartite graph is matching covered if and only if it admits an ear decomposition.
    \label{thm : odd ear decomp}
\end{thm}

There is also a generalization of the above theorem for nonbipartite \mcg{}s due to Lov\'asz and Plummer \cite{lopl86}. However, we do not describe this here since our work focuses on \bmcg{}s.

\subsection{Minimality, bounds and corresponding notions of extremality}
\label{subsec : bounds}

A \mcg{} is \i{minimal} if deletion of any edge results in a graph that is not matching covered. We use $\c{H}$ to denote the class of \mbmcg{}s. Using the ear decomposition theory, Lov\'asz and Plummer \cite{lopl86} proved that each member of $\c{H}$ has at least $m-n+2$ pairwise nonadjacent $2$-edges --- where a \i{$2$-edge} is an edge whose each end has degree two. Thus, such a graph has at least $2(m-n+2)$ degree two vertices. Furthermore, for each of these invariants, one may also deduce lower bounds solely in terms of $n$; in particular, that each member of $\c{H}$ has at least $\frac{n+10}{6}$ $2$-edges and at least $\frac{n}{2}+2$ vertices of degree two; see Corollaries~\ref{cor : bound een} and~\ref{cor : bound evn}. They also proved an upper bound --- namely, that each member of $\c{H}$, distinct from $C_4$, has at most $\frac{3n-6}{2}$ edges. 

Each of the bounds stated in the above paragraph leads to a notion of extremality for \mbmcg{}s. As we have five different bounds, we have five notions of extremality as defined in Table~\ref{tab: extremality notions}.  For instance, $\c{H}_2$ denotes the class of \i{\evm{}} \mbmcg{}s --- that is, those members of~$\c{H}$ that satisfy $|V_2|=2(m-n+2)$ where $V_2$ is the set of degree two vertices. Likewise, we use $E_2$ to denote the set of $2$-edges. 

\begin{table}[htb]
    \centering
    \begin{tabular}{|c|c|c|}
        \hline
        Class & Property & Notation  \\\hline
        \i{\eem} & $|E_2|=m-n+2$ & $\c{H}_0$\\\hline
        \i{\een} & $|E_2|=\frac{n+10}{6}$ & $\c{H}_1$\\\hline
        \i{\evm} & $|V_2|=2(m-n+2)$ & $\c{H}_2$\\\hline
        \i{\evn} & $|V_2|=\frac{n}{2}+2$ & $\c{H}_3$\\\hline
        \i{\ee} & $|E|=\frac{3n-6}{2}$ & $\c{H}_4$\\\hline
    \end{tabular}
    \caption{Definitions of different notions of extremality within $\c{H}$}
    \label{tab: extremality notions}
\end{table}

In this paper, we completely characterize all of the five classes of ``extremal" \mbmcg{}s in terms of special trees. Our characterizations are reminiscent of similar characterization(s) of  ``extremal" \mtc{}s --- where minimality is defined with respect to edge deletion; we briefly describe some of these results below before stating our characterizations.  

Dirac \cite{gad67} proved that a \mtc{} has at least $\frac{n+4}{3}$ vertices of degree two. A \mtc{} is \i{extremal} if it satisfies this lower bound with equality. Oxley~\cite{jgo82} gave a generation theorem for \emtc{}s. Karpov \cite{dvk16} then characterized this class of graphs in terms of special trees. In particular, they proved that a graph is extremal minimal \kC{2} if and only if it can be obtained from two copies of a tree, each of whose non-leaves is of degree three, by identifying the corresponding leaves as per some fixed isomorphism; see Figure~\ref{fig: karpov} for an example.

\begin{figure}[htb]
    \begin{subfigure}{0.49\linewidth}
        \begin{drawing}{1}

            \foreach \y in {1}{
                \foreach \x in {-3,-2,...,3}{
                    \draw[color=green] (\x,0.5*\y) circle (4pt);
                }
                \foreach \x in {-2,-0.5,1,2}{
                    \draw[color=red] (\x,1.5*\y) circle (4pt);
                }
                    \draw[color=red] (-0.5,2.5*\y) circle (4pt);

                \draw (-3,0.5*\y)--(-2,1.5*\y)--(-2,0.5*\y);
                \draw (-1,0.5*\y)--(-0.5,1.5*\y)--(0,0.5*\y);
                \draw (1,0.5*\y)--(1,1.5*\y);
                \draw (2,0.5*\y)--(2,1.5*\y)--(3,0.5*\y);

                \draw (-2,1.5*\y)--(-0.5,2.5*\y)--(-0.5,1.5*\y);
                \draw (-0.5,2.5*\y)--(1,1.5*\y)--(2,1.5*\y);
            }
            
            \draw (-3,0.5)node[wnode]{}(-2,1.5)(-2,0.5)node[wnode]{};
            \draw (-2,1.5)node[bnode]{}(-0.5,2.5)(-0.5,1.5)(-1,0.5)node[wnode]{};
            \draw (0,0.5)node[wnode]{}(-0.5,1.5)node[bnode]{}++(0,1)node[wnode]{}(1,1.5)(1,0.5)node[wnode]{};
            \draw (1,1.5)node[bnode]{}(2,0.5)node[bnode]{};
            \draw (2,1.5)node[wnode]{}(3,0.5)node[bnode]{};

 
        \end{drawing}
    \end{subfigure}
    \begin{subfigure}{0.49\linewidth}
        \begin{drawing}{1}

            \foreach \y in {-1,1}{
                \foreach \x in {-3,-2,...,3}{
                    \draw[color=green] (\x,0*\y) circle (4pt);
                }
                \foreach \x in {-2,-0.5,1,2}{
                    \draw[color=red] (\x,1*\y) circle (4pt);
                }
                    \draw[color=red] (-0.5,2*\y) circle (4pt);

                \draw (-3,0*\y)--(-2,1*\y)--(-2,0*\y);
                \draw (-1,0*\y)--(-0.5,1*\y)--(0,0*\y);
                \draw (1,0*\y)--(1,1*\y);
                \draw (2,0*\y)--(2,1*\y)--(3,0*\y);

                \draw (-2,1*\y)--(-0.5,2*\y)--(-0.5,1*\y);
                \draw (-0.5,2*\y)--(1,1*\y)--(2,1*\y);

                \draw (-3,0*\y)node[wnode]{}(-2,1*\y)(-2,0)node[wnode]{};
                \draw (-2,1*\y)node[bnode]{}(-0.5,2*\y)(-0,1*\y)(-1,0*\y)node[wnode]{};
                \draw (0,0)node[wnode]{}(-0.5,1*\y)node[bnode]{}++(0,1*\y)node[wnode]{}(1,1*\y)(1,0)node[wnode]{};
                \draw (1,1*\y)node[bnode]{}(2,0)node[bnode]{};
                \draw (2,1*\y)node[wnode]{}(3,0)node[bnode]{};
            }

 
        \end{drawing}
    \end{subfigure}
    \caption{illustration of Karpov's characterization of \emtc{}s}
    \label{fig: karpov}
\end{figure}

A tree is said to be a \i{\htree{}} if all of its non-leaves have degree at least three, or equivalently, if it has no vertices of degree two; see Figure~\ref{fig: ex evm} for an example. We call them halin trees because a \i{Halin graph} is a (planar) graph that is obtained from a planar embedding of such a tree by adding a cycle each of whose edges joins two leaves that appear consecutively in the cyclic order (as per the planar embedding). Halin graphs, first introduced by Halin \cite{rh71}, as examples of minimal \kC{3} graphs, have been studied extensively in the literature. Halin trees are precisely the ``homeomorphically irreducible" trees; finding such  spanning trees in a cubic graph is a well-studied problem; see \cite{ref-hist}. A \htree{} is \i{cubic} if the corresponding Halin graph is cubic. To put it differently, a tree is a cubic \htree{} if all of its non-leaves have degree exactly three; see Figure~\ref{fig: karpov} for an example. Thus, the aforementioned Karpov's characterization of \emtc{}s draws a bijection between this graph class and cubic halin trees.       

\subsection{Characterizations of $\c{H}_2,\c{H}_3$ and $\c{H}_4$ using halin trees}
\label{subsec : h2 characterization}
We are now ready to state our characterizations of ``extremal" \mbmcg{}s --- that are similar to Karpov's characterization. However, in our case, identifying the corresponding leaves is not the correct operation. 

For two disjoint copies of a \htree{}, say $T$ and $T'$, let $H$ be the (bipartite) graph obtained from $T\cup T'$ by adding a matching each of whose edges joins a leaf of $T$ with the corresponding leaf of $T'$ as per some fixed isomorphism between $T$ and $T'$. We say that $H$ is obtained from $T$ and $T'$ by \i{\isojoin{}}, or simply, that $H$ is obtained from $T$ by \isojoin{}. 

Note that, since $K_2$ is the smallest \htree{}, $C_4$ is the smallest graph that may be obtained using the operation defined above. The following observation is easily proved and it implies that, for every other graph obtained using this operation, one may recover the two copies of the \htree{} by simply deleting the set of $2$-edges. 

\begin{prop}
    If a graph $H$, distinct from $C_4$, is obtained from a \htree{} $T$ by \isojoin{}, then $H-E_2$ has precisely two components, each of which is isomorphic to $T$. \qed
    \label{prop : t cup t'}
\end{prop}

We first focus on the \evm{} class $\c{H}_2$ (see Table~\ref{tab: extremality notions}) and obtain the following characterization establishing a bijection between \evm{} \bmcg{}s and \htree{}s. 

\begin{thm}{\sc[Main Theorem: Characterization of $\c{H}_2$]}\\
    A graph is a \evm{} \mbmcg{} if and only if it is obtained from a \htree{} by \isojoin{}.
    \label{thm : evm}
\end{thm}

\begin{figure}[htb]
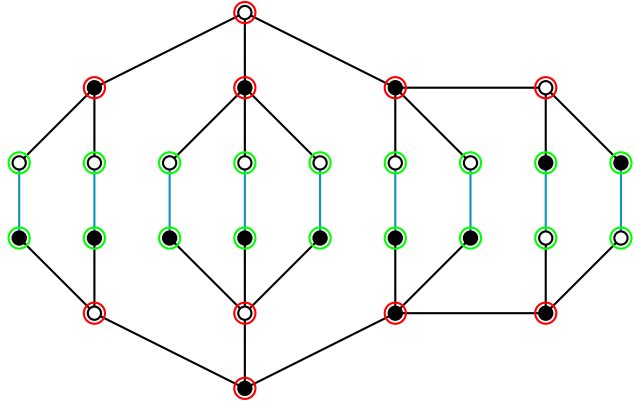

        
    \centering


        \centering
        \begin{drawing}{1}
            \foreach \x in {-3,-2,...,5}{
                \draw[color=cyan!70!black] (\x,-0.5)--(\x,0.5);
            }
            
            \draw (-3,0.5)node[wnode]{}--(-2,1.5)--(-2,0.5)node[wnode]{};
            \draw (-2,1.5)node[bnode]{}--(0,2.5)--(0,1.5)--(-1,0.5)node[wnode]{};
            \draw (0,0.5)node[wnode]{}--(0,1.5);
            \draw (1,0.5)node[wnode]{}--(0,1.5)node[bnode]{}++(0,1)node[wnode]{}--(2,1.5)--(2,0.5)node[wnode]{};
            \draw (3,0.5)node[wnode]{}--(2,1.5)node[bnode]{}--(4,1.5)--(4,0.5)node[bnode]{};
            \draw (4,1.5)node[wnode]{}--(5,0.5)node[bnode]{};

            \draw (-3,-0.5)node[bnode]{}--(-2,-1.5)--(-2,-0.5)node[bnode]{};
            \draw (-2,-1.5)node[wnode]{}--(-0,-2.5)--(-0,-1.5)--(-1,-0.5)node[bnode]{};
            \draw (0,-0.5)node[bnode]{}--(0,-1.5);
            \draw (1,-0.5)node[bnode]{}--(-0,-1.5)node[wnode]{}++(0,-1)node[bnode]{}--(2,-1.5)--(2,-0.5)node[bnode]{};
            \draw (3,-0.5)node[bnode]{}--(2,-1.5)node[wnode]{}--(4,-1.5)--(4,-0.5)node[wnode]{};
            \draw (4,-1.5)node[bnode]{}--(5,-0.5)node[wnode]{};

            \foreach \y in {-1,1}{
                \foreach \x in {-3,-2,...,5}{
                    \draw[color=green] (\x,0.5*\y) circle (4pt);
                }
                \foreach \x in {-2,0,2,4}{
                    \draw[color=red] (\x,1.5*\y) circle (4pt);
                }
                \draw[color=red] (0,2.5*\y) circle (4pt);
            }

        \end{drawing}
        
        \caption{a member of $\c{H}_2$ --- deleting the $2$-edges results in two copies of a \htree{}}
        \label{fig: ex evm}
\end{figure}

Figure~\ref{fig: ex evm} shows an illustration of the above theorem. In this figure, and relevant figures henceforth, we adopt the following color conventions: the red vertices have degree three or more, the green vertices are of degree two, and the cyan edges are precisely the $2$-edges.


Thereafter, we will use our Main Theorem (\ref{thm : evm}) to derive characterizations of the extremal classes $\c{H}_3$ and $\c{H}_4$ as stated below; see Figure~\ref{fig: ex evn and ee} for an example of each.

\begin{thm}{\sc[Characterization of $\c{H}_3$]}\\
    A graph is a \evn{} \mbmcg{} if and only if it is obtained from a cubic \htree{} by \isojoin{}.
    \label{thm : evn}
\end{thm}

By a \i{star}, we mean $K_{1,p}$ where $p\geqslant 2$. Observe that these, except $K_{1,2}$, comprise a restricted subclass of \htree{}s; they play a crucial role in our characterization of $\c{H}_4$ stated below.

\begin{thm}{\sc[Characterization of $\c{H}_4$]}\\
    A graph is an \ee{} \mbmcg{} if and only if it is obtained from a star by \isojoin{}.
    \label{thm : ee}
\end{thm}

\begin{figure}[htb]
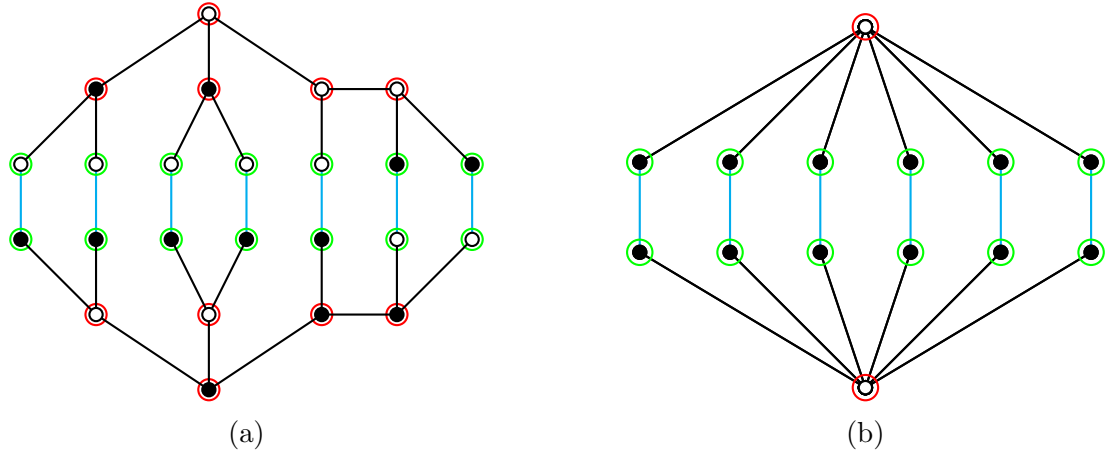

    \centering
    \begin{subfigure}{0.49\linewidth}
        \centering
        \begin{drawing}{1}
            \foreach \x in {-3,-2,...,3}{
                \draw[color=cyan] (\x,-0.5)--(\x,0.5);
            }

            \foreach \y in {-1,1}{
                \foreach \x in {-3,-2,...,3}{
                    \draw[color=green] (\x,0.5*\y) circle (4pt);
                }
                \foreach \x in {-2,-0.5,1,2}{
                    \draw[color=red] (\x,1.5*\y) circle (4pt);
                }
                    \draw[color=red] (-0.5,2.5*\y) circle (4pt);

                \draw (-3,0.5*\y)--(-2,1.5*\y)--(-2,0.5*\y);
                \draw (-1,0.5*\y)--(-0.5,1.5*\y)--(0,0.5*\y);
                \draw (1,0.5*\y)--(1,1.5*\y);
                \draw (2,0.5*\y)--(2,1.5*\y)--(3,0.5*\y);

                \draw (-2,1.5*\y)--(-0.5,2.5*\y)--(-0.5,1.5*\y);
                \draw (-0.5,2.5*\y)--(1,1.5*\y)--(2,1.5*\y);
            }
            
            \draw (-3,0.5)node[wnode]{}(-2,1.5)(-2,0.5)node[wnode]{};
            \draw (-2,1.5)node[bnode]{}(-0.5,2.5)(-0.5,1.5)(-1,0.5)node[wnode]{};
            \draw (0,0.5)node[wnode]{}(-0.5,1.5)node[bnode]{}++(0,1)node[wnode]{}(1,1.5)(1,0.5)node[wnode]{};
            \draw (1,1.5)node[bnode]{}(2,0.5)node[bnode]{};
            \draw (2,1.5)node[wnode]{}(3,0.5)node[bnode]{};

            \draw (-3,-0.5)node[bnode]{}(-2,-1.5)(-2,-0.5)node[bnode]{};
            \draw (-2,-1.5)node[wnode]{}(-0.5,-2.5)(-0.5,-1.5)(-1,-0.5)node[bnode]{};
            \draw (0,-0.5)node[bnode]{}(-0.5,-1.5)node[wnode]{}++(0,-1)node[bnode]{}(1,-1.5)(1,-0.5)node[bnode]{};
            \draw (1,-1.5)node[wnode]{}(2,-0.5)node[wnode]{};
            \draw (2,-1.5)node[bnode]{}(3,-0.5)node[wnode]{};
 
        \end{drawing}
        \caption{}
        \label{fig:ex-ec}
    \end{subfigure}
    \begin{subfigure}{0.49\linewidth}
        \centering
        \begin{drawing}{1.2}
            \foreach \x in {-2,...,3}{
                \draw[color=cyan] (\x,-0.5)--(\x,0.5);
            }

            \foreach \x in {-2,...,3}{
                \draw (\x,0.5)node[bnode]{}--(0.5,2)node[wnode]{};
                \draw (\x,-0.5)node[bnode]{}--(0.5,-2)node[bnode]{};
            }
            \foreach \y in {-1,1}{
                \foreach \x in {-2,...,3}{
                    \draw (\x,0.5*\y)node[bnode]{}--(0.5,2*\y)node[wnode]{};
                    \draw[color=green] (\x,0.5*\y) circle (4pt);
                }
                \draw[color=red] (0.5,2*\y) circle (4pt);
            }
        \end{drawing}
        \caption{}
        \label{fig:ex-ed}
    \end{subfigure}
    \caption{(a) a member of $\c{H}_3$: deleting the $2$-edges results in two copies of a cubic \mbox{\htree{}}
    \strut \hspace{1.75cm}(b) a member of $\c{H}_4$: deleting the $2$-edges results in two copies of a star}
    \label{fig: ex evn and ee}
\end{figure}

In the next subsection, we prove the easier implications of the characterizations stated above (that is, Theorems~\ref{thm : evm},~\ref{thm : evn} and~\ref{thm : ee}). 

\subsection{Proofs of easier implications}

We begin by proving that any graph obtained from a nontrivial tree by \isojoin{} is a \mbmcg{}. We shall find ear decompositions useful to establish the \mc{} property; however, for minimality, we need an easy observation pertaining to cuts.

Given a graph $G:=(V,E)$, for any $W\subseteq V$, the \i{cut} comprising the edges joining $W$ and its complement $\overline{W}:=V- W$ is denoted by $\partial (W)$. A cut $\partial(W)$ is \i{trivial} if either $W$ or $\overline{W}$ comprises at most one vertex, and \i{nontrivial} otherwise. A \i{$k$-cut} refers to a cut of cardinality~$k$. We use abbreviated notations $\partial(v):=\partial(\{v\})$ for a vertex $v$, and $\partial(H):=\partial(V(H))$ for a subgraph $H$.

An edge $e$ of a \mcg{} $G$ is \i{removable} if $G-e$ is also \mc{}. Thus, a \mcg{} is minimal if and only if it has no removable edges. Since \mcg{}s are $2$-connected, we immediately observe the following. 
\begin{lem}
    In a \mcg{}, any edge that participates in a $2$-cut is not removable. \hfill{$\qed$}
    \label{lem : nonremovable edge}
\end{lem}
The above lemma is the aforementioned observation that helps us in establishing minimality in the proof of the following result; we shall find it useful later as well.

\begin{prop}
    Any graph obtained from a nontrivial tree by \isojoin{} is a \mbmcg{}.
    \label{prop : mbmcg TcupT'}
\end{prop}
\begin{proof} 
    Let $G$ be a graph obtained from two copies of a nontrivial tree, say $T$ and $T'$, by \isojoin{}. We will first prove, by induction on the order of $T$, that $G$ is a \bmcg{}. If $T$ is a path then $G$ is a even cycle, and we are done. 

    Now, let $T$ be a tree that is not a path and let $v$ denote a leaf. Let $P$ be the maximal path starting at $v$, each of whose internal vertices has degree two in $T$. Let $u$ be the end of $P$ distinct from $v$. Observe that $d_T(u)\geqslant3$. Let $v'P'u'$ be the path corresponding to $vPu$ in~$T'$. Let $H:=G-V(P-u)- V(P'-u')$. Observe that $H$ is obtained from the tree $T-V(P-u)$ by \isojoin{}. By the induction hypothesis, $H$ is a \bmcg{}. Note that $G$ may be obtained from $H$ by adding the ear $uPvv'P'u'$. By the ear decomposition theorem (\ref{thm : odd ear decomp}), $G$ is a \bmcg{}.

    Now, we prove minimality. Observe that any edge $e$ of $G$, whose each end has degree at least three, either belongs to $T$ or belongs to $T'$. Adjust notation so that $e\in T$. Let $e'$ be the copy of $e$ in $T'$. Note that $\{e,e'\}$ is a $2$-cut of $G$. Thus, by Lemma~\ref{lem : nonremovable edge},  $e$ is not removable. We thus infer that $G$ is minimal.
\end{proof}

\begin{figure}[htb]
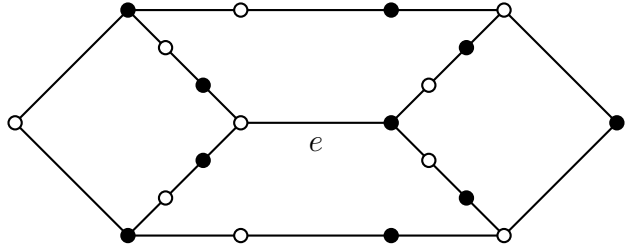

    \centering
    \begin{drawing}{1}
        \draw (-2.5,1.5)--(2.5,1.5)--(1,0)--(2.5,-1.5)--(-2.5,-1.5)--(-1,0)--(-2.5,1.5);
        
        \draw (-1.5,0.5)node[bnode]{}(-2,1)node[wnode]{};
        \draw (-1,1.5)node[wnode]{}(1,1.5)node[bnode]{};
        \draw (1.5,0.5)node[wnode]{}(2,1)node[bnode]{};
        \draw (-1.5,-0.5)node[bnode]{}(-2,-1)node[wnode]{};
        \draw (-1,-1.5)node[wnode]{}(1,-1.5)node[bnode]{};
        \draw (1.5,-0.5)node[wnode]{}(2,-1)node[bnode]{};
        
        \draw (-1,0)--(1,0);
        
        \draw (0,0)node[below,nodelabel]{$e$};
    
        \draw (-1,0)node[wnode]{}(-2.5,1.5)node[bnode]{}--(-4,0)node[wnode]{}--(-2.5,-1.5)node[bnode]{};

        \draw (1,0)node[bnode]{}(2.5,1.5)node[wnode]{}--(4,0)node[bnode]{}--(2.5,-1.5)node[wnode]{};
    \end{drawing}
    \caption{a counterexample to the converse of Proposition~\ref{prop : mbmcg TcupT'}}
    \label{fig: counterexample mbmcg}
\end{figure}

We note that the converse of the above proposition does not hold. The graph $G$, shown in Figure~\ref{fig: counterexample mbmcg}, is a \mbmcg{} that can not be obtained from a tree  by \isojoin{}. To see this, suppose to the contrary that $G$ is obtained from two copies of a tree, say $T$ and $T'$, by \isojoin{}. Then, as per the proof of Proposition~\ref{prop : mbmcg TcupT'}, the labeled edge $e$ belongs either to $T$ or to $T'$ as both ends of $e$ have degree greater than two; furthermore, $e$ belongs to a $2$-cut of $G$. However, $G-e$ is \kC{2}. This is a contradiction. We shall revisit this example in Section~\ref{lem : basic prop of embmcg}.  

In light of Proposition~\ref{prop : mbmcg TcupT'}, in order to prove the easier implications of Theorems~\ref{thm : evm},~\ref{thm : evn} and~\ref{thm : ee}, we simply need to argue that if we choose a tree $T$ from the corresponding class (as in that theorem's statement), then the resulting \mbmcg{} $G$ (obtained by \isojoin{}) satisfies the desired extremality notion. We first note that if $T$ is the star $K_{1,2}$ then $G=C_6$ is an \ee{} \mbmcg{}. All of the remaining cases are handled by the following theorem --- which we prove using straightforward counting arguments.

\begin{thm}
    Any graph $G$ obtained from a \htree{} $T$ by \isojoin{} is a \mbmcg{} that is \evm{}. Furthermore:
    \begin{enumerate}
        \item if $T$ is a cubic \htree{} then $G$ is also \evn{}, whereas
        \item if $T$ is a star then $G$ is also \ee{}. 
    \end{enumerate}
    \label{thm : embmcg TcupT'}
\end{thm}
\begin{proof}
    Let $G:=(V,E)$ be a graph obtained from two copies of a \htree{}, say $T$ and $T'$, by \isojoin{}. By Proposition~\ref{prop : mbmcg TcupT'}, $G$ is a \mbmcg{}. It remains to prove ``extremality". We use $V_2$ and $E_2$ to denote $V_2(G)$ and $E_2(G)$, respectively. 

    Since $T$ is a \htree{}, the non-leaves of $T$ and $T'$ comprise those vertices of $G$ that have degree at least three, whereas the leaves of $T$ and $T'$ comprise the remaining vertices of~$G$ (that is, its vertices of degree two). In particular, there are no $2$-edges in $T\cup T'$, and $T\cup T'= G-E_2$.

    Now, $|E(T)|=|V(T)|-1$. So, $|E(T\cup T')|=n-2$. On the other hand, $|E(T\cup T')|=|E|-|E_2|=m-|E_2|$. Thus, $|E_2|=m-n+2$. Since $E_2$ is a matching of $G$, and the corresponding matched vertices comprise the set $V_2$, we get $|V_2|=2|E_2|=2(m-n+2)$. Hence, $G$ is \evm{}. It remains to prove statements \i{(i)} and \i{(ii)}.
    
    First, suppose that $T$ is a cubic \htree{}. So, $T\cup T'$ has $|V_2|$ leaves and $n-|V_2|$ vertices of degree three. By handshaking lemma, $3(n-|V_2|)+|V_2|=2|E(T\cup T')|=2(n-2)$. By simplifying, we get $|V_2|=\frac{n}{2}+2$; whence $G$ is \evn{}.
    
    Now, suppose that $T$ is a star $K_{1,p}$, where $p\geqslant3$. Then, $n=2(p+1)$. Observe that $m=3p=\frac{3n-6}{2}$. Ergo, $G$ is \ee{}.
\end{proof}

We now switch our attention to the remaining ``extremal" classes $\c{H}_0$ and $\c{H}_1$.

\subsection{Characterizations of $\c{H}_0$ and $\c{H}_1$}
\label{subsec : h0 and h1}

In general, members of $\c{H}_0$ and $\c{H}_1$ may not be obtained from a tree by \isojoin{}, unlike the members of $\c{H}_2,\c{H}_3$ and $\c{H}_4$. To see this, we discuss a couple of examples that are shown in Figure~\ref{fig: examples not isojoin}; using the ear decomposition theorem (\ref{thm : odd ear decomp}) and Lemma~\ref{lem : nonremovable edge}, the reader may easily verify that both graphs belong to $\c{H}$. By counting, one may infer these graphs belong to $\c{H}_0$ and $\c{H}_1$,  respectively. For the graph shown in Figure~\ref{subfig : ex h0 1}, deleting all of the $2$-edges (displayed in blue) results in precisely two components that are trees but nonisomorphic. On the other hand, for the graph shown in Figure~\ref{subfig : ex h0 2}, deleting all of the $2$-edges results in two isomorphic trees; however, the resultant graph is not obtained by \isojoin{} operation. Interestingly, we are able to reduce members of $\c{H}_0$~and~$\c{H}_1$ to members of $\c{H}_2$ and $\c{H}_4$, respectively, using the notion of ``\retract" that we describe next. 

\begin{figure}[htb]
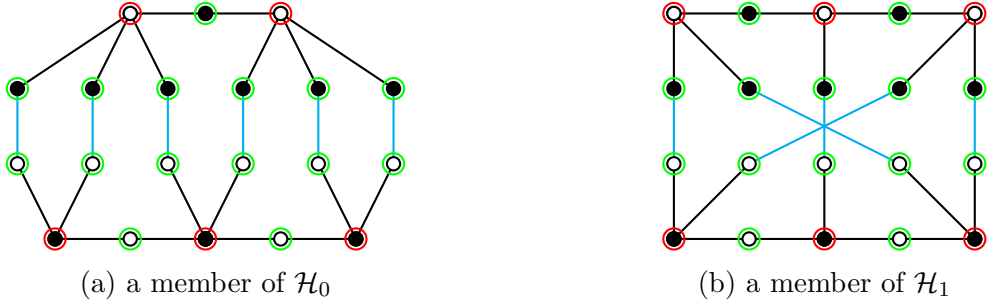

    \centering
    \begin{subfigure}{0.49\linewidth}
        \centering
        \begin{drawing}{1}
            \foreach \x in {-2.5,-1.5,...,2.5}{
                \draw[color=cyan] (\x,1)--(\x,2);
            }
            \foreach \x in {1,-1}{
                \draw (0,0)node[bnode]{}--++(1*\x,0)node[wnode]{}--++(1*\x,0)node[bnode]{}--++(-0.5*\x,1)node[wnode]{}++(0,1)node[bnode]{}--++(-0.5*\x,1)node[wnode]{}--++(-1*\x,0)node[bnode]{};
                \draw (0,0)node[bnode]{}--++(0.5*\x,1)node[wnode]{}++(0,1)node[bnode]{}--++(0.5*\x,1)node[wnode]{}--++(1.5*\x,-1)node[bnode]{}++(0,-1)node[wnode]{}--++(-0.5*\x,-1)node[bnode]{};
            }

            \foreach \x in {-2.5,-1.5,...,2.5}{
                \draw[color=green] (\x,1) circle (4pt);
                \draw[color=green] (\x,2) circle (4pt);
            }

            \foreach \x in {-2,0,2}{  
                \draw[color=red] (\x,0) circle (4pt);
            }          
            \foreach \x in {-1,1}{  
                \draw[color=green] (\x,0) circle (4pt);
            }
            \draw[color=red] (-1,3) circle (4pt);
            \draw[color=red] (1,3) circle (4pt);
            \draw[color=green] (0,3) circle (4pt);

        \end{drawing}
        \caption{a member of $\c{H}_0$}
        \label{subfig : ex h0 1}
    \end{subfigure}
    \begin{subfigure}{0.49\linewidth}
        \centering
        \begin{drawing}{1}
            \foreach \x in {0,4}{
                \draw[color=cyan] (\x,1)--(\x,2);
            }
            \foreach \x in {1,2,3}{
                \draw[color=cyan] (\x,1)--(4-\x,2);
            }
            
            \draw (0,0)node[bnode]{}--++(1,0)node[wnode]{}--++(1,0)node[bnode]{}--++(1,0)node[wnode]{}--++(1,0)node[bnode]{}--++(0,1)node[wnode]{}++(0,1)node[bnode]{}--++(0,1)node[wnode]{}--++(-1,0)node[bnode]{}--++(-1,0)node[wnode]{}--++(-1,0)node[bnode]{}--++(-1,0)node[wnode]{}--++(0,-1)node[bnode]{}++(0,-1)node[wnode]{}--++(0,-1)node[bnode]{};
    
            \draw (2,0)node[bnode]{}--++(0,1)node[wnode]{}++(0,1)node[bnode]{}--++(0,1)node[wnode]{};
            \draw (0,0)node[bnode]{}--++(1,1)node[wnode]{}++(2,1)node[bnode]{}--++(1,1)node[wnode]{};
            \draw (4,0)node[bnode]{}--++(-1,1)node[wnode]{}++(-2,1)node[bnode]{}--++(-1,1)node[wnode]{};

            \foreach \x in {0,1,...,4}{
                \draw[color=green] (\x,1) circle (4pt);
                \draw[color=green] (\x,2) circle (4pt);
            }
            \foreach \y in {0,3}{
                \foreach \x in {0,2,4}{  
                    \draw[color=red] (\x,\y) circle (4pt);
                }
                \foreach \x in {1,3}{  
                    \draw[color=green] (\x,\y) circle (4pt);
                }
            }
            
        \end{drawing}
        \caption{a member of $\c{H}_1$}
        \label{subfig : ex h0 2}
    \end{subfigure}
    \caption{extremal graphs that may not be obtained from a tree by \isojoin{}}
    \label{fig: examples not isojoin}
\end{figure}

Let $G$ be a graph and $v$ be a vertex of degree two that has two distinct neighbours, say $v_1$~and~$v_2$. Let $G'$ denote the graph obtained from~$G$ by contracting the two edges $vv_1$ and $vv_2$. We say that $G'$ is obtained from $G$ by \i{bicontraction} (of $v$), and we denote it by $G':=G/v$; see Figure~\ref{fig: bicontraction and bisplitting}. Furthermore, if each of $v_1$ and $v_2$ has degree at least three in $G$, we say that $G'$ is obtained from $G$ by \i{\bicontract{}} (of $v$); see Figure~\ref{fig: restricted bicontraction and restricted bisplitting}. Note that \bicontract{} may be performed if and only if the subgraph of $G$ induced by its degree two vertices has an isolated vertex.

We define the \i{\retract} of $G$ as the graph $\widehat{G}$ obtained by repeatedly applying the \bicontract{} operation until the resulting graph has the property that the subgraph induced by its degree two vertices has no isolated vertex. Observe that given a graph~$G$, its \retract{} $\widehat{G}$ is unique; Figure~\ref{fig: ex h0} shows an example.

Using the notion of \retract{}, we are able to relate the extremal class $\c{H}_0$ with the extremal class $\c{H}_2$, as stated below.

\begin{thm}{\sc[Characterization of $\c{H}_0$]}\\
    A graph $G$, distinct from $C_4$, is a \eem{} \mbmcg{} if and only if its \retract{} $\widehat{G}$ is a \evm{} \mbmcg{}.
    \label{thm : eem}
\end{thm}

\begin{figure}[htb]
    \centering
    \begin{subfigure}{0.49\linewidth}
        \centering
        \begin{drawing}{1}
            \foreach \x in {-3.5,-2.5,...,2.5,3.5}{
                \draw[color=cyan] (\x,1)--(\x,2);
            }
            \foreach \x in {1,-1}{
                \draw (3*\x,0)--++(-0.5*\x,1)node[bnode]{}++(0,1)node[wnode]{}--++(-0.5*\x,1);
                \draw (2*\x,0)--++(1*\x,0)node[wnode]{}--++(0.5*\x,1)node[bnode]{}++(0,1)node[wnode]{}--++(-1.5*\x,1)node[bnode]{}--++(-1*\x,0);
                \draw (0,0)node[bnode]{}--++(1*\x,0)node[wnode]{}--++(1*\x,0)node[bnode]{}--++(-0.5*\x,1)node[wnode]{}++(0,1)node[bnode]{}--++(-0.5*\x,1)node[wnode]{}--++(-1*\x,0)node[bnode]{};
                \draw (0,0)node[bnode]{}--++(0.5*\x,1)node[wnode]{}++(0,1)node[bnode]{}--++(0.5*\x,1)node[wnode]{};
            }

            \foreach \x in {-3.5,-2.5,...,2.5,3.5}{
                \draw[color=green] (\x,1) circle (4pt);
                \draw[color=green] (\x,2) circle (4pt);
            }

            \foreach \x in {-3,-2,0,2,3}{  
                \draw[color=red] (\x,0) circle (4pt);
            }          
            \foreach \x in {-1,1}{  
                \draw[color=green] (\x,0) circle (4pt);
            }
            \draw[color=red] (-1,3) circle (4pt);
            \draw[color=red] (1,3) circle (4pt);
            \draw[color=red] (-2,3) circle (4pt);
            \draw[color=red] (2,3) circle (4pt);
            \draw[color=green] (0,3) circle (4pt);

        \end{drawing}
        \caption{$G\in\c{H}_0$}
        \label{subfig : h_0}
    \end{subfigure}
    \begin{subfigure}{0.49\linewidth}
        \centering
        \begin{drawing}{1}
            \foreach \x in {-3.5,-2.5,...,2.5,3.5}{
                \draw[color=cyan] (\x,1)--(\x,2);
            }
            
            \foreach \x in {1,-1}{
                \draw (2*\x,0)--++(1.5*\x,1)node[bnode]{}++(0,1)node[wnode]{}--++(-1.5*\x,1);
                \draw (0,0)--++(2*\x,0)node[wnode]{}--++(0.5*\x,1)node[bnode]{}++(0,1)node[wnode]{}--++(-0.5*\x,1)node[bnode]{}--++(-2*\x,0);
                
            }
            
            \foreach \x in {-1.5,-0.5,...,1.5}{
                \draw (\x,2)node[bnode]{}--(0,3)node[wnode]{};
                \draw (\x,1)node[wnode]{}--(0,0)node[bnode]{};
            }

            \foreach \x in {-3.5,-2.5,...,2.5,3.5}{
                \draw[color=green] (\x,1) circle (4pt);
                \draw[color=green] (\x,2) circle (4pt);
            }
            \draw[color=red] (0,0) circle (4pt);
            \draw[color=red] (-2,0) circle (4pt);
            \draw[color=red] (2,0) circle (4pt);
            \draw[color=red] (0,3) circle (4pt);
            \draw[color=red] (-2,3) circle (4pt);
            \draw[color=red] (2,3) circle (4pt);            
        \end{drawing}
        \caption{$\widehat{G}\in \c{H}_2$}
        \label{subfig : h_2}
    \end{subfigure}
    
    \caption{an illustration of Theorem~\ref{thm : eem} that relates $\c{H}_0$ and $\c{H}_2$}
    \label{fig: ex h0}
\end{figure}

To see an example of the above theorem, consider the graph $G$ shown in Figure~\ref{subfig : h_0} and its \retract{} $\widehat{G}$ shown in Figure~\ref{subfig : h_2}. The reader may verify using Theorem~\ref{thm : odd ear decomp} and Lemma~\ref{lem : nonremovable edge} that $G$ and $\widehat{G}$ belong to $\c{H}$. Note that $G\in \c{H}_0$ as $|E_2(G)|=m-n+2$. On the other hand, $\widehat{G}\in \c{H}_2$ as $|V_2(\widehat{G})|=2(m-n+2)$.

Likewise, using the notion of \retract{}, we are able to relate the extremal class $\c{H}_1$ with the extremal class $\c{H}_4$, as stated below.

\begin{thm}{\sc[Characterization of $\c{H}_1$]}\\
    A graph $G$ is a \een{} \mbmcg{} if and only if its \retract{} $\widehat{G}$ is an \ee{} \mbmcg{} and $\Delta(G)=3$.
    \label{thm : een}
\end{thm}

\begin{figure}[htb]
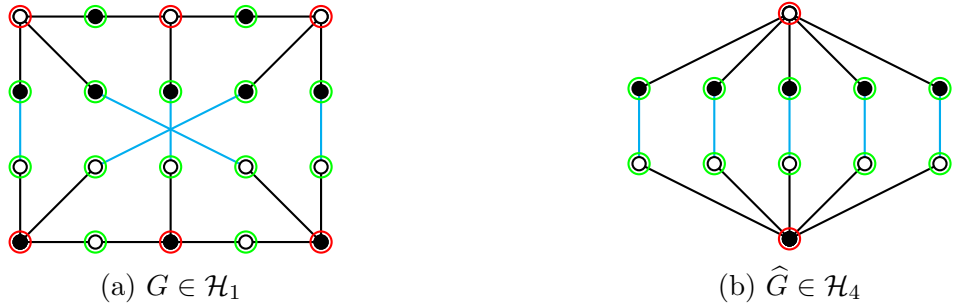

    \centering
    \begin{subfigure}{0.49\linewidth}
        \centering
        \begin{drawing}{1}
            \foreach \x in {0,4}{
                \draw[color=cyan] (\x,1)--(\x,2);
            }
            \foreach \x in {1,2,3}{
                \draw[color=cyan] (\x,1)--(4-\x,2);
            }
            
            \draw (0,0)node[bnode]{}--++(1,0)node[wnode]{}--++(1,0)node[bnode]{}--++(1,0)node[wnode]{}--++(1,0)node[bnode]{}--++(0,1)node[wnode]{}++(0,1)node[bnode]{}--++(0,1)node[wnode]{}--++(-1,0)node[bnode]{}--++(-1,0)node[wnode]{}--++(-1,0)node[bnode]{}--++(-1,0)node[wnode]{}--++(0,-1)node[bnode]{}++(0,-1)node[wnode]{}--++(0,-1)node[bnode]{};
    
            \draw (2,0)node[bnode]{}--++(0,1)node[wnode]{}++(0,1)node[bnode]{}--++(0,1)node[wnode]{};
            \draw (0,0)node[bnode]{}--++(1,1)node[wnode]{}++(2,1)node[bnode]{}--++(1,1)node[wnode]{};
            \draw (4,0)node[bnode]{}--++(-1,1)node[wnode]{}++(-2,1)node[bnode]{}--++(-1,1)node[wnode]{};

            \foreach \x in {0,1,...,4}{
                \draw[color=green] (\x,1) circle (4pt);
                \draw[color=green] (\x,2) circle (4pt);
            }
            \foreach \y in {0,3}{
                \foreach \x in {0,2,4}{  
                    \draw[color=red] (\x,\y) circle (4pt);
                }
                \foreach \x in {1,3}{  
                    \draw[color=green] (\x,\y) circle (4pt);
                }
            }
            
        \end{drawing}
        \caption{$G\in\c{H}_1$}
        \label{subfig : h_1}
    \end{subfigure}
    \begin{subfigure}{0.49\linewidth}
        \centering
        \begin{drawing}{1}
            \foreach \x in {-2,-1,...,2}{
                \draw[color=cyan] (\x,1)--(\x,2);
                \draw (\x,2)node[bnode]{}--(0,3)node[wnode]{};
                \draw (\x,1)node[wnode]{}--(0,0)node[bnode]{};
            }

            \foreach \x in {-2,-1,...,2}{
                \draw[color=green] (\x,1) circle (4pt);
                \draw[color=green] (\x,2) circle (4pt);
            }
            \draw[color=red] (0,0) circle (4pt);
            \draw[color=red] (0,3) circle (4pt);            
        \end{drawing}
        \caption{$\widehat{G}\in \c{H}_4$}
        \label{subfig : h_4}
    \end{subfigure}
    
    \caption{an illustration of Theorem~\ref{thm : een} that relates $\c{H}_1$ and $\c{H}_4$}
    \label{fig: ex h1}
\end{figure}

To see an example of the above theorem, consider the graph $G$ shown in Figure~\ref{subfig : h_1} and its \retract{} $\widehat{G}$ shown in Figure~\ref{subfig : h_4}. The reader may verify using Theorem~\ref{thm : odd ear decomp} and Lemma~\ref{lem : nonremovable edge} that $G$ and $\widehat{G}$ belong to $\c{H}$. Note that $G\in \c{H}_1$ as $|E_2(G)|=\frac{n+10}{6}$. On the other hand, $\widehat{G}\in \c{H}_4$ as $|E(\widehat{G})|=\frac{3n-6}{2}$.

We have thus stated our characterizations for all of the extremal classes shown in Table~\ref{tab: extremality notions}. We now proceed to discuss further relations between them. 

\subsection{Containment poset for the extremal classes}

In order to discuss the containment relationships between different extremal classes, we shall find it convenient to exclude the cycle graphs. Following Lucchesi and Murty~\cite{lumu24}, for an element $t$ and a set $S$, we use $S-t$ to denote the set obtained from $S$ by deleting $t$. It follows from the definitions of the extremal classes (see Table~\ref{tab: extremality notions}) that neither of $\c{H}_0$ and $\c{H}_1$ contains any cycle graph, that $C_4$ is the only cycle in $\c{H}_2$ as well as in $\c{H}_3$, whereas $C_6$ is the only cycle in $\c{H}_4$. In light of this, we let $\c{H}_2^*:=\c{H}_2-C_4$, $\c{H}_3^*:=\c{H}_3-C_4$ and $\c{H}_4^*:=\c{H}_4-C_6$.


\begin{figure}[htb]
\centering
    \begin{minipage}{0.49\linewidth}
    \centering
    \begin{drawing}{1.5}
            \foreach \x in {-1,0,1}{
                \draw[color=cyan] (\x,1)--(\x,2);
                \draw (\x,2)node[bnode]{}--(0,3)node[wnode]{};
                \draw (\x,1)node[wnode]{}--(0,0)node[bnode]{};
            }

            \foreach \x in {-1,0,1}{
                \draw[color=green] (\x,1) circle (2.5pt);
                \draw[color=green] (\x,2) circle (2.5pt);
            }
            \draw[color=red] (0,0) circle (2.5pt);
            \draw[color=red] (0,3) circle (2.5pt);            
    \end{drawing}
    \caption{The $\Theta$ graph}
    \label{fig: theta}
    \end{minipage}
    \begin{minipage}{0.49\linewidth}
    \centering
    
    \begin{drawing}{1}
        \node [circle, draw] at (0,-3)(theta){$\{\Theta\}$};
        \node [circle, draw] at (0,0)(h3){$\c{H}_3^*$};
        \node [circle, draw] at (1.9,-1.5)(h4){$\c{H}_4^*$};
        \node [circle, draw] at (1.9,1.5)(h2){$\c{H}_2^*$};
        \node [circle, draw] at (-2.2,0)(h1){$\c{H}_1$};
        \node [circle, draw] at (0,3)(h0){$\c{H}_0$};

        \draw [-stealth](theta)--(h3);
        \draw [-stealth](theta)--(h4);
        \draw [-stealth](theta)--(h1);
        \draw [-stealth](h3)--(h2);
        \draw ($(h3)!0.5!(h2)$)++(0.1,-0.1)node[above left,nodelabel]{\ref{cor : h3 subset h2}};
        \draw [-stealth](h4)--(h2);
        \draw ($(h4)!0.5!(h2)$)++(-0.1,-0.0)node[right,nodelabel]{\ref{cor : h4 subset h2}};
        \draw [-stealth](h1)--(h0);
        \draw ($(h1)!0.5!(h0)$)++(0.1,-0.1)node[above left,nodelabel]{\ref{cor : h1 subset h0}};
        \draw [-stealth](h2)--(h0);
        \draw ($(h2)!0.5!(h0)$)++(-0.1,-0.1)node[above right,nodelabel]{\ref{cor : h2 subset h0}};
    \end{drawing}
    \caption{Containment poset}
    \label{fig: poset}
    \end{minipage}
\end{figure}

Figure~\ref{fig: poset} shows the containment poset for all of the extremal classes. For instance, the arrow from $\c{H}_1$ to $\c{H}_0$, labelled as~\ref{cor : h1 subset h0}, is to be read as follows: Corollary~\ref{cor : h1 subset h0} states that $\c{H}_1\subset \c{H}_0$. Additionally, our characterizations imply that $\c{H}_1\cap\c{H}_2=\c{H}_3\cap \c{H}_4=\{\Theta\}$ where $\Theta$ is the graph shown in Figure~\ref{fig: theta}; see Corollaries~\ref{cor : h4 cap h3 = theta} and~\ref{cor : h1 cap h2 = theta}.


\subsection{Organization of the rest of this paper}

In the rest of the paper (except the last section), we will prove the difficult directions of Theorems~\ref{thm : evm},~\ref{thm : evn} and~\ref{thm : ee}, and we will also prove Theorems~\ref{thm : eem} and~\ref{thm : een}. 

We first focus on the class $\c{H}_2$. In order to prove the Main Theorem (\ref{thm : evm}), we need some known properties of \bmcg{}s and \mbmcg{}s; these are discussed in Sections~\ref{sec : ear decomp} and~\ref{sec : mbmcg}. In particular, in Section~\ref{sec : mbmcg}, we state and prove the lower bound on the number of $2$-edges due to Lov\'asz and Plummer \cite{lopl86}, and we deduce an easy lower bound on the number of degree two vertices. In Section~\ref{sec : embmcg}, we enforce equality in its proof to infer some useful properties of members of $\c{H}_2$ --- one of which is the existence of ``balanced" $2$-cuts. We then develop an induction tool using balanced $2$-cuts in Section~\ref{sec : 2-cut induction tool}. In Section~\ref{sec : main theorem}, we prove the Main Theorem (\ref{thm : evm}).

In Section~\ref{sec : other classes}, we address the other notions of extremality. In Subsections~\ref{subsec : h3} and~\ref{subsec : h4}, we first state and prove the bounds corresponding to the classes $\c{H}_3$ and $\c{H}_4$, respectively; thereafter, we enforce equality in their proofs to deduce their characterizations (Theorems~\ref{thm : evn} and~\ref{thm : ee}) using the Main Theorem (\ref{thm : evm}). In Subsection~\ref{subsec : h0}, we prove our characterization of~$\c{H}_0$ (Theorem~\ref{thm : eem}). Lastly, in Subsection~\ref{subsec : h1}, we prove our characterization of $\c{H}_1$ (Theorem~\ref{thm : een}) using the characterization of $\c{H}_0$.

Finally, in Section \ref{sec : kex}, we consider problems similar to those solved by Corollaries \ref{cor : bound evm}, \ref{cor : bound evn} and \ref{cor : bound ee} in the context of ``minimal \ex{k}" bipartite graphs. We first discuss the bounds established by Lou \cite{l99} and then proceed to conjecture stronger bounds; see {Conjectures \ref{conj : evm}, \ref{conj : evn} and \ref{conj : ee}}. In the rest of that section, we provide evidence for our conjectures by constructing tight examples that are straightforward generalizations of the ones that appear in the \ex{1} case. 

\section{Ear decomposition theory}
\label{sec : ear decomp}
Recall that a subgraph $H$ of a graph $G$ is {conformal} if $G-H$ is matchable. It is easy to see that, in a \mc{} graph, there is a conformal cycle containing any two adjacent edges. As an immediate consequence, we get the following.

\begin{prop}
    Let $G$ be a \mcg{}, let $u$ be a vertex of degree at least three and let $e\in \partial(u)$. Then there is a conformal cycle $C$ such that $u\in C$ but $e\notin C$. \qed
    \label{prop : C_u}
    
\end{prop}

In fact, Little \cite{litt74} proved the stronger statement that any two edges (not necessarily adjacent) lie in a conformal cycle; we will not require this. 

Recall that an {ear of a subgraph $H$ in a graph $G$} is an odd path of $G$ whose ends are in~$H$ but is otherwise disjoint from~$H$. We will now discuss the proof of the Ear Decomposition Theorem (\ref{thm : odd ear decomp}) as it appears in Lov\'asz and Plummer \cite[Theorem 4.1.6]{lopl86}. In particular, we extract a couple of lemma statements that are implicit in their proof. The first of these shows that ``adding an ear" preserves the \mc{} property. For a subgraph $H$ and an ear $P$ of $H$, the subgraph $H\cup P$ is said to be obtained by adding the ear $P$ to $H$.

\begin{lem}
    Let $G$ be a bipartite graph, $H$ be a \mc{} subgraph and $P$ be an ear of $H$ in $G$. Then $H\cup P$ is \mc{}. 
    \label{lem : H cup P is mc}
\end{lem}

The second lemma establishes the existence of an ear of any conformal matching covered proper subgraph $H$ with the additional constraint that any specified edge ``sticking out of~$H$" is included in the ear.

\begin{lem}
    Let $G$ be a \bmcg{}, $H$ be any conformal \mc{} subgraph and $e\notin H$ denote an edge that has at least one end in $H$. Then there exists an ear of $H$, say $P_e$, containing $e$ such that $H\cup P_e$ is a conformal \mc{} subgraph of $G$.
    \label{lem : P_e}
\end{lem}

The reader may verify that the above two lemmas imply the following stronger version of the Ear Decomposition Theorem (\ref{thm : odd ear decomp}) that is also stated in Lov\'asz and Plummer \cite[pg 126]{lopl86}. 



\begin{thm}
    For any \cmcs{} $H$ of a \bmcg{} $G$, there exists an ear decomposition of $G$ in which $H$ appears as one of the subgraphs.\qed
    \label{thm : strong ear decomp}
\end{thm}

In the next section, we discuss the proofs of the lower bounds on the number of $2$-edges due to Lov\'asz and Plummer \cite[Theorem 4.2.7]{lopl86}. These proofs rely heavily on the ear decomposition theory, and we will need some of their details to prove our results.



\section{Lower bounds established by Lov\'asz and Plummer}
\label{sec : mbmcg}

Recall that, for a \mcg{} $G$, an edge $e$ is {removable} if $G-e$ is \mc{}. Thus, $G$ is minimal if and only if it has no removable edges. It follows from Theorem~\ref{thm : odd ear decomp} that, given any ear decomposition of a \bmcg{} $G$, every trivial ear (that is, an ear comprising a single edge) is a removable edge in $G$. This observation, coupled with Theorem~\ref{thm : strong ear decomp}, yields the following; see Lov\'asz and Plummer \cite[Theorem 4.2.1]{lopl86}.


\begin{lem}
    Every \cmcs{} of a \mbmcg{} is an induced subgraph.\hfill{$\qed$}
    \label{lem : cmcs is induced}
\end{lem}





In order to establish their lower bound on the number of $2$-edges in a \mbmcg{}, Lov\'asz and Plummer proved a lower bound for each element (the starting conformal cycle and each ear) of an ear sequence. The following is their lemma for ears. 

\begin{lem}
    Let $G$ be a \mbmcg{} and let $(C,P_1,P_2,\dots ,P_r)$ be any ear sequence. Then each ear $P_i$ contains a $2$-edge of $G$.
    \label{lem : 2-edge in ear}
\end{lem}

Note that the above lemma (see Lov\'asz and Plummer \cite[Lemma 4.2.4]{lopl86}) refers to degrees in the original graph $G$. Adding a nontrivial ear $P_i$ clearly creates a $2$-edge in the resultant subgraph. However, one may add ears later thereby increasing the degrees of internal vertices of $P_i$ and possibly destroying all $2$-edges of $P_i$ in the final graph $G$. The above lemma states that at least one such $2$-edge survives in $G$. We now switch our attention to conformal cycles.

Let $G$ be a \mbmcg{} that is not a cycle. Using Lemma~~\ref{lem : 2-edge in ear} and some standard ear decomposition tricks, one may infer that each conformal cycle $C$ of~$G$ contains a pair of $2$-edges that are separated in $C$ by vertices that have degree at least three (in $G$). Interestingly, the same holds even for cycles that are not conformal; see Lov\'asz and Plummer \cite[Corollary 4.2.5]{lopl86}. 

A matching $M$ is \i{induced} if the subgraph induced by $V(M)$ contains precisely the edges of $M$. The facts stated in the above paragraph, along with Lemma~\ref{lem : cmcs is induced}, imply the following (that is easily verified in the case in which the graph itself is a cycle of order six or more). 

\begin{cor}
    In a \mbmcg{}, distinct from $C_4$, each conformal cycle contains a pair of $2$-edges that comprise an induced matching. 
    \label{cor : 2-edge in cycle}
\end{cor}

We now invoke Lemma~\ref{lem : 2-edge in ear} and Corollary~\ref{cor : 2-edge in cycle}, and the fact that the length of any ear sequence is $m-n+1$, to deduce the following lower bound result that is a mild strengthening of what is stated in Lov\'asz and Plummer; see \cite[Theorem 4.2.7]{lopl86}. 

\begin{thm}
    A \mbmcg{}, distinct from $C_4$, has an induced matching of size at least $m-n+2$, each of whose members is a $2$-edge.
    \label{thm : lower bound on 2-edges}
\end{thm}
\begin{proof}
    Let $G$ be a \mbmcg{} and $(C,P_1,\dots, P_r)$ be any ear sequence. By Corollary~\ref{cor : 2-edge in cycle}, the conformal cycle $C$ has a pair of $2$-edges, say $e$ and $e'$, that comprise an induced matching of $G$. By Lemma~\ref{lem : 2-edge in ear}, each ear $P_i$ has a $2$-edge, say $e_i$. Observe that each end of $e$ and $e'$ has both its neighbours in $C$. Likewise, each end of $e_i$ has both its neighbours in $P_i$. Furthermore, the ends of each ear $P_i$ have degree at least three in $G$. All of these observations imply that the set $F:=\{e,e',e_1,e_2,\dots, e_r\}$ is an induced matching of $G$ and each of its members is a $2$-edge. Also, $r=m-n$. Thus, $|F|=m-n+2$.
\end{proof}

The following lower bound on the number of vertices of degree two comes for free. 

\begin{cor} In a \mbmcg{}, $|V_2|\geqslant 2(m-n+2)$.
\label{cor : bound evm}
\end{cor}
\begin{proof}
    The statement clearly holds for $C_4$. Now, let $G$ be a \mbmcg{}, distinct from $C_4$, and let $F$ denote a maximum induced matching, each of whose members is a $2$-edge. Clearly, $|V_2|\geqslant 2|F|$. By Theorem~\ref{thm : lower bound on 2-edges}, $|F|\geqslant m-n+2$. Thus, $|V_2|\geqslant 2(m-n+2)$. 
\end{proof}

Throughout the next three sections, for the sake of brevity, we shall use the term \i{extremal} to refer to \evm{} --- that is, those \mbmcg{}s that satisfy the bound in Corollary~\ref{cor : bound evm} with equality. 


\section{Properties of extremal graphs}
\label{sec : embmcg}

In this section, we establish some useful properties of this class including the ``balanced \mbox{$2$-cut"} property that is crucial for our inductive proof of Theorem~\ref{thm : evm} that characterizes this class. To this end, we shall enforce equality in the proofs of Corollary~\ref{cor : bound evm} and Theorem~\ref{thm : lower bound on 2-edges}.





      

\begin{lem}
    Every \embmcg{}, distinct from $C_4$, satisfies the following:
    \begin{enumerate}
        \item $E_2$ is a perfect matching of $G[V_2]$ and $|E_2|=m-n+2$,
        \item each ear in any ear decomposition has exactly one $2$-edge, and
        \item each conformal cycle has precisely two $2$-edges.
    \end{enumerate}
    
    \label{lem : basic prop of embmcg}
\end{lem}
\begin{proof}
Let $G$ be an \embmcg{}, distinct from $C_4$, and let $F$ denote an induced matching of maximum cardinality, each of whose members is a $2$-edge, as in the proof of Corollary~\ref{cor : bound evm}. Since, by definition, $|V_2|=2(m-n+2)$, equality holds everywhere in the proof of Corollary~\ref{cor : bound evm}. In particular, $|F|=m-n+2$ and $|V_2|=2|F|$; thus $V_2=V(F)$. Observe that, since $F$ is an induced matching, $E_2=F$. Consequently, $E_2$ is a perfect matching of $G[V_2]$ and $|E_2|=m-n+2$. This proves \i{(i)}.

Now, let $(C,P_1,\dots ,P_r)$ be any ear sequence of $G$. Since $|F|=m-n+2$, equality holds everywhere in the proof of Theorem~\ref{thm : lower bound on 2-edges}. Furthermore, using the fact that $|E_2|=m-n+2$, we infer that the conformal cycle $C$ contributes exactly two $2$-edges, and each ear $P_i$ contributes precisely one $2$-edge, to the set $E_2$. This proves \i{(ii)} and \i{(iii)}. 
\end{proof}

Using Lemma~\ref{lem : basic prop of embmcg} \i{(i)}, we immediately deduce the following containment.
\begin{cor}
    $\c{H}_2-C_4\subset \c{H}_0$. \qed
    \label{cor : h2 subset h0}
\end{cor}
We now derive, using the above lemma and simple counting arguments, a few additional properties of this extremal class. In order to do so, we need some notation and terminology that we define next.

Following Lov\'asz and Plummer \cite{lopl86}, we call an edge (of a graph $G$) a \i{\tedge{}} if each of its ends has degree at least three, and we use $E_3(G)$ to denote the set comprising these edges. In the same spirit, we use $V_3(G)$ to denote the set of vertices of degree at least three. Technically, it should be denoted by $V_{\geqslant3}(G)$. However, throughout this paper, we distinguish between vertices of degree two and the rest --- that is, vertices of degree three or more. So, for the sake of brevity, we drop ``$\geqslant$" from the notation. Finally, we use $E_{3,2}(G)$ to denote the set of those edges that join a vertex of degree two with a vertex of degree at least three. As usual, if the graph $G$ is clear from the context, we shall drop $G$ from these notations. Note that $E_{3,2}=\partial(V_3)=\partial(V_2)$. 

Now, let $G$ be an \embmcg{} distinct from~$C_4$. By definition,  $|V_3|=n-|V_2|=3n-2m-4$. By Lemma~\ref{lem : basic prop of embmcg}, each vertex of degree two has exactly one neighbour of degree at least three; consequently, $|E_{3,2}|=|V_2|=2(m-n+2)$. Lastly, $|E_{3}|=m-|E_2|-|E_{3,2}|=3n-2m-6$. Also, since $|V_3|=3n-2m-4$, it follows that $|E_3|=|V_3|-2$. This proves the following.  

\begin{cor}
    Every \embmcg{}, distinct from $C_4$, satisfies the following:
    \begin{enumerate}
        \item $|V_3|=3n-2m-4$,
        \item $|E_{3,2}|=2m-2n+4$, and 
        \item $|E_{3}|=3n-2m-6=|V_3|-2$.\hfill{$\qed$}        
    \end{enumerate}
    
    \label{cor : counting prop of embmcg}
\end{cor}

We shall now use Lemma~\ref{lem : basic prop of embmcg} to prove the ``balanced $2$-cut" property of \embmcg{}s that was alluded to at the beginning of this section. 

\subsection{Balanced $2$-cut property}
For a cut $\partial(W)$ of a bipartite graph $G[A,B]$, we classify the edges of $\partial(W)$ into two types, depending upon the color classes of their ends in $W$. We use $\partial^A(W)$ to denote those edges of $\partial(W)$ whose end in $W$ belongs to $A$. The set $\partial^B(W)$ is defined analogously. The cut $\partial(W)$ is \i{balanced} if $|\partial^A(W)|=|\partial^B(W)|$. In our work, balanced $2$-cuts play a crucial role. Note that a $2$-cut $\partial(W)$ is balanced if $|\partial^A(W)|=|\partial^B(W)|=1$. We are now ready to prove the balanced $2$-cut property stated below.

\begin{thm}{\sc[Balanced $2$-cut Property]}\\
    In an \embmcg{}, each \tedge{} participates in a balanced $2$-cut with some other \tedge{}.
    \label{thm : 2-cut property}
\end{thm}
\begin{proof}
    Let $G[A,B]$ denote an \embmcg{} and let $e:=ab$ denote any \tedge{} so that $a\in A$. We intend to show the existence of another \tedge{}, say $f$, such that $\{e,f\}$ is a balanced $2$-cut.
    
    By Proposition~\ref{prop : C_u}, there is a conformal cycle $C$ such that $a\in C$ and $e\notin C$. Now, let $H$ be a maximal conformal matching covered subgraph of $G$ such that $a\in H$ and $e\notin H$. By Lemma~\ref{lem : cmcs is induced}, $H$ is induced; thus $b\notin H$. In particular, $e\in \partial^A(H)$. We shall demonstrate that $\partial(H)$ is the desired balanced $2$-cut. In order to do so, we state and prove specific claims; within their proofs, we shall invoke ear decomposition results (namely, Lemma~\ref{lem : P_e} and Theorem~\ref{thm : strong ear decomp}) several times without mentioning it explicitly.

    \begin{stat}
    $\partial^A(H)=\{e\}$.
    \end{stat}

    \begin{proof}
    Suppose not; then there is an edge $f\neq e$ in $\partial^A(H)$. Now, there is an ear $P_f$ of $H$ containing $f$. Observe that $e\notin P_f$ since $P_f$ is odd. Thus, $H':=H\cup P_f$ is a larger conformal matching covered subgraph of $G$ such that $a\in H'$ and $e\notin H'$; this contradicts the maximality of $H$. This proves that $\partial^A(H)=\{e\}$. 
    \end{proof}

    In the proofs of the next two claims, every instance of $2$-edge refers to the condition of being a $2$-edge in the graph $G$ similar to our earlier discussion (after the proof of Lemma~\ref{lem : 2-edge in ear}).
    \begin{stat}
    $|\partial^B(H)|= 1$.
    \label{stat : one bw edge}
    \end{stat}
    \begin{proof}
    Since $G$ is \kC{2}, $|\partial^B(H)|\geqslant1$. Now suppose that $f_1$ and $f_2$ are distinct edges in $\partial^B(H)$ and let $b_1$ and $b_2$ denote their (not necessarily distinct) ends in $H$, respectively. Now, there is an ear $P'$ of $H$ containing $f_1$ (which must also contain $e$). Furthermore, there is an ear $P_2$ of $H\cup P'$ containing $f_2$. Let $x\in A$ be the end of $P_2$ distinct from $b_2$ (see Figure~\ref{subfig: h cup p+p1 cup p2}). Observe that $x\notin H$ since $\partial^A(H)=\{e\}$ and $e\in P'$. Thus, $x$ is an internal vertex of $P'$. Consequently, $x$ splits $P'$ into two paths, say $P:=aP'x$ and $P_1:=xP'b_1$. Note that $P_2$ is an ear of $H\cup (P+P_1)=H\cup P'$; likewise, $P_1$ is an ear of $H\cup (P+P_2)$. By Lemma~\ref{lem : basic prop of embmcg} \i{(ii)}, each of $P'$, $P_1$ and $P_2$ has exactly one $2$-edge. Since $P_1$ is a subpath of $P'$, we conclude that $P$ is free of $2$-edges. Furthermore, by Lemma~\ref{lem : basic prop of embmcg} \i{(i)}, each vertex of $P$ has degree at least three in $G$.

    \begin{figure}[htb]
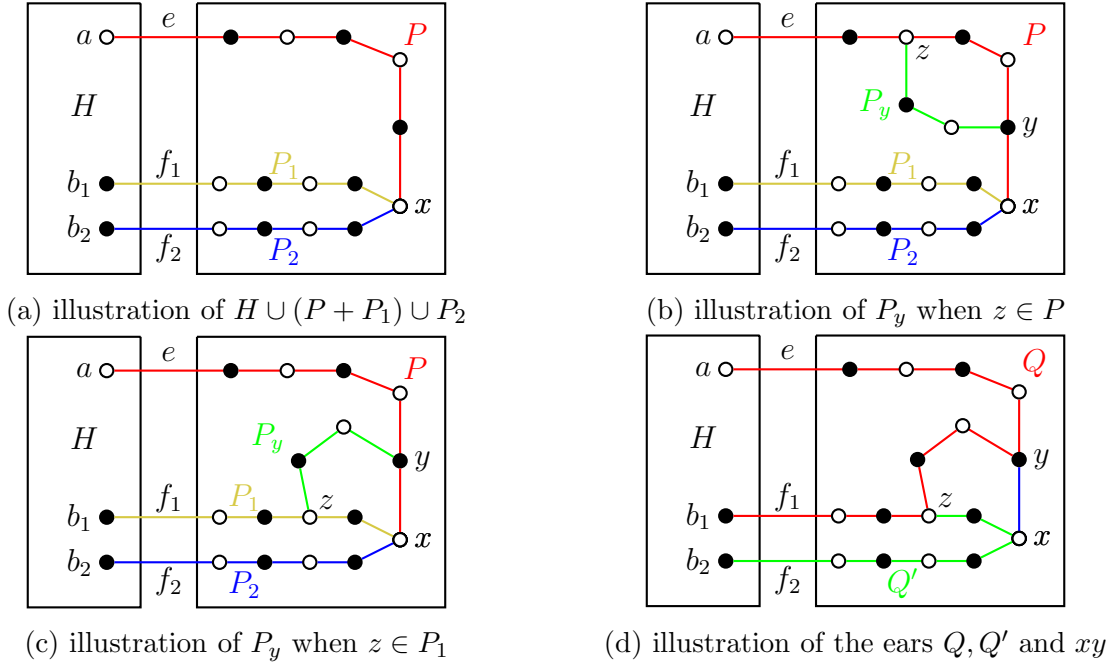

        \centering
        \begin{subfigure}{0.49\linewidth}
            \centering
            \begin{drawing}{1.5}
                \draw (-0.5,1.2)--(-1.5,1.2)--(-1.5,-1.2)--(-0.5,-1.2)--(-0.5,1.2);
            \draw (0,1.2)--(2.2,1.2)--(2.2,-1.2)--(0,-1.2)--(0,1.2);
            
            \draw (-0.8,0.9)node(a)[wnode]{}edge++(1.1,0)node[bnode]{}edge++(0.5,0)node[wnode]{}edge++(0.5,0)node[bnode]{}edge++(0.5,-0.2)node[wnode]{}edge++(0,-0.6)node[bnode]{}edge++(0,-0.7)node[wnode](x){}node[right,nodelabel]{$x$};
    
            \let\edgecolor\blueedge
            \draw (-0.8,-0.8)node(b2)[bnode]{}edge++(1,0)node[wnode]{}edge++(0.4,0)node[bnode]{}edge++(0.4,0)node[wnode]{}edge++(0.4,0)node[bnode]{}edge(x)node[wnode]{};
    
            \let\edgecolor\yellowedge       
            \draw (-0.8,-0.4)node(b1)[bnode]{}edge++(1,0)node[wnode]{}edge++(0.4,0)node[bnode]{}edge++(0.4,0)node[wnode]{}edge++(0.4,0)node[bnode]{}edge(x)node[wnode]{};
    
            \draw (x)node[wnode]{}node[right,nodelabel]{$x$};
            \draw (a)node[left,nodelabel]{$a$};
            \draw (b1)node[left,nodelabel]{$b_1$};
            \draw (b2)node[left,nodelabel]{$b_2$};
            \draw (a)++(0.55,0.15)node[nodelabel]{$e$};
            \draw (b1)++(0.55,0.17)node[nodelabel]{$f_1$};
            \draw (b2)++(0.55,-0.17)node[nodelabel]{$f_2$};
            
            \draw (-1,0.3)node[nodelabel]{$H$};
            \draw[color=yellow!80!black] (0.5,-0.24)node[right,nodelabel]{$P_1$};
            \draw[color=blue] (0.5,-1)node[right,nodelabel]{$P_2$};
            \draw[color=red] (1.7,0.7)node[above right,nodelabel]{$P$};
            \end{drawing}
            \caption{illustration of $H\cup (P+P_1)\cup P_2$}
            \label{subfig: h cup p+p1 cup p2}
        
        \end{subfigure}
        \begin{subfigure}{0.49\linewidth}
            \centering
            \begin{drawing}{1.5}
            \draw (-0.5,1.2)--(-1.5,1.2)--(-1.5,-1.2)--(-0.5,-1.2)--(-0.5,1.2);
            \draw (0,1.2)--(2.2,1.2)--(2.2,-1.2)--(0,-1.2)--(0,1.2);
            
            \draw (-0.8,0.9)node(a)[wnode]{}edge++(1.1,0)node[bnode]{}edge++(0.5,0)node[wnode](z){}edge++(0.5,0)node[bnode]{}edge++(0.4,-0.2)node[wnode]{}edge++(0,-0.6)node[bnode](y){}edge++(0,-0.7)node[wnode](x){}node[right,nodelabel]{$x$};
    
            \let\edgecolor\blueedge
            \draw (-0.8,-0.8)node(b2)[bnode]{}edge++(1,0)node[wnode]{}edge++(0.4,0)node[bnode]{}edge++(0.4,0)node[wnode]{}edge++(0.4,0)node[bnode]{}edge(x)node[wnode]{};
    
            \let\edgecolor\yellowedge       
            \draw (-0.8,-0.4)node(b1)[bnode]{}edge++(1,0)node[wnode]{}edge++(0.4,0)node[bnode]{}edge++(0.4,0)node[wnode]{}edge++(0.4,0)node[bnode]{}edge(x)node[wnode]{};
    
            \let\edgecolor\greenedge
            \draw (y)edge++(-0.5,0)node[wnode]{}edge++(-0.4,0.2)node[bnode]{}edge(z);
    
            \draw (x)node[wnode]{}node[right,nodelabel]{$x$};
            \draw (a)node[left,nodelabel]{$a$};
            \draw (b1)node[left,nodelabel]{$b_1$};
            \draw (b2)node[left,nodelabel]{$b_2$};
            \draw (y)node[right,nodelabel]{$y$};
            \draw (z)++(-0.05,0.05)node[below right,nodelabel]{$z$};
            
            \draw (-1,0.3)node[nodelabel]{$H$};
            \draw[color=yellow!80!black] (0.5,-0.24)node[right,nodelabel]{$P_1$};
            \draw[color=blue] (0.5,-1)node[right,nodelabel]{$P_2$};
            \draw[color=red] (1.7,0.7)node[above right,nodelabel]{$P$};
            \draw[color=green] (0.8,0.3)node[left,nodelabel]{$P_y$};
            \draw (a)++(0.55,0.15)node[nodelabel]{$e$};
            \draw (b1)++(0.55,0.17)node[nodelabel]{$f_1$};
            \draw (b2)++(0.55,-0.17)node[nodelabel]{$f_2$};
            \end{drawing}
            \caption{illustration of $P_y$ when $z\in P$}
            \label{subfig: z in P}
        \end{subfigure}
        \begin{subfigure}{0.49\linewidth}
            \centering
            \begin{drawing}{1.5}
                \draw (-0.5,1.2)--(-1.5,1.2)--(-1.5,-1.2)--(-0.5,-1.2)--(-0.5,1.2);
                \draw (0,1.2)--(2.2,1.2)--(2.2,-1.2)--(0,-1.2)--(0,1.2);

                \draw (-0.8,0.9)node[wnode](a){}edge++(1.1,0)node[bnode]{}edge++(0.5,0)node[wnode]{}edge++(0.5,0)node[bnode]{}edge++(0.5,-0.2)node[wnode]{}edge++(0,-0.6)node[bnode](y){}edge++(0,-0.7)node[wnode](x){}node[right,nodelabel]{$x$};

                \let\edgecolor\blueedge
                \draw (-0.8,-0.8)node(b2)[bnode]{}edge++(1,0)node[wnode]{}edge++(0.4,0)node[bnode]{}edge++(0.4,0)node[wnode]{}edge++(0.4,0)node[bnode]{}edge(x)node[wnode]{};
        
                \let\edgecolor\yellowedge       
                \draw (-0.8,-0.4)node(b1)[bnode]{}edge++(1,0)node[wnode]{}edge++(0.4,0)node[bnode]{}edge++(0.4,0)node[wnode](z){}edge++(0.4,0)node[bnode]{}edge(x)node[wnode]{};

                \let\edgecolor\greenedge
                \draw (y)edge++(-0.5,0.3)node[wnode]{}edge++(-0.4,-0.3)node[bnode]{}edge(z);
        
                \draw (x)node[wnode]{}node[right,nodelabel]{$x$};
                \draw (a)node[left,nodelabel]{$a$};
                \draw (b1)node[left,nodelabel]{$b_1$};
                \draw (b2)node[left,nodelabel]{$b_2$};
                \draw (y)node[right,nodelabel]{$y$};
                \draw (z)++(-0.05,-0.05)node[above right,nodelabel]{$z$};
                
                \draw (-1,0.3)node[nodelabel]{$H$};
                \draw[color=yellow!80!black] (0.15,-0.23)node[right,nodelabel]{$P_1$};
                \draw[color=blue] (0.15,-1)node[right,nodelabel]{$P_2$};
                \draw[color=red] (1.7,0.7)node[above right,nodelabel]{$P$};
                \draw[color=green] (0.9,0.3)node[left,nodelabel]{$P_y$};
                \draw (a)++(0.55,0.15)node[nodelabel]{$e$};
                \draw (b1)++(0.55,0.17)node[nodelabel]{$f_1$};
                \draw (b2)++(0.55,-0.17)node[nodelabel]{$f_2$};
            \end{drawing}
            \caption{illustration of $P_y$ when $z\in P_1$}
            \label{subfig: z in p_1}    
        \end{subfigure}
        \begin{subfigure}{0.49\linewidth}
            \centering
            \begin{drawing}{1.5}
                \draw (-0.5,1.2)--(-1.5,1.2)--(-1.5,-1.2)--(-0.5,-1.2)--(-0.5,1.2);
                \draw (0,1.2)--(2.2,1.2)--(2.2,-1.2)--(0,-1.2)--(0,1.2);

                \draw (-0.8,0.9)node[wnode](a){}edge++(1.1,0)node[bnode]{}edge++(0.5,0)node[wnode]{}edge++(0.5,0)node[bnode]{}edge++(0.5,-0.2)node[wnode]{}edge++(0,-0.6)node[bnode](y){};

                \draw (y)edge++(-0.5,0.3)node[wnode]{}edge++(-0.4,-0.3)node[bnode]{}edge(z);
                
                \draw (-0.8,-0.4)node(b1)[bnode]{}edge++(1,0)node[wnode]{}edge++(0.4,0)node[bnode]{}edge++(0.4,0)node[wnode](z){};

                \let\edgecolor\greenedge
                \draw (z)edge++(0.4,0)node[bnode]{}edge(x)node[wnode]{};

                \draw (-0.8,-0.8)node(b2)[bnode]{}edge++(1,0)node[wnode]{}edge++(0.4,0)node[bnode]{}edge++(0.4,0)node[wnode]{}edge++(0.4,0)node[bnode]{}edge(x)node[wnode]{};

                \let\edgecolor\blueedge
                \draw (y)edge++(0,-0.7)node[wnode](x){}node[right,nodelabel]{$x$};
                
                
                \draw (x)node[wnode]{}node[right,nodelabel]{$x$};
                \draw (a)node[left,nodelabel]{$a$};
                \draw (b1)node[left,nodelabel]{$b_1$};
                \draw (b2)node[left,nodelabel]{$b_2$};
                \draw (y)node[right,nodelabel]{$y$};
                \draw (z)++(-0.05,-0.05)node[above right,nodelabel]{$z$};
                
                \draw (-1,0.3)node[nodelabel]{$H$};
                \draw[color=green] (0.5,-1)node[right,nodelabel]{$Q'$};
                \draw[color=red] (1.7,0.7)node[above right,nodelabel]{$Q$};
                \draw (a)++(0.55,0.15)node[nodelabel]{$e$};
                \draw (b1)++(0.55,0.17)node[nodelabel]{$f_1$};
                \draw (b2)++(0.55,-0.17)node[nodelabel]{$f_2$};
            \end{drawing}
            \caption{illustration of the ears $Q,Q'$ and $xy$}
            \label{subfig: ears Q, Q' and xy}    
        \end{subfigure}
        \caption{illustrations for the proof of statement~\ref{stat : one bw edge}}
        \label{fig:enter-label}
    \end{figure}

    Now, let $y$ be the neighbour of $x$ in $P$. As discussed above, $d_G(y)\geqslant3$. So, there is an ear $P_y$ of $H\cup (P+P_1)\cup P_2$ starting from $y$. Let $z$ be the other end of $P_y$. Observe that $z\notin H$ as $\partial^A(H)=\{e\}$. 
    
    Now, suppose that $z\in P$ (see Figure~\ref{subfig: z in P}). Note that the \cmcs{} $H\cup (P+P_1)\cup P_y$ can also be obtained from $H$ by adding the ears $aPzP_yyxP_1b_1$ and $yPz$ in that order. However, $yPz$ has no $2$-edges since, as noted earlier, each vertex of $P$ has degree at least three in $G$. This contradicts Lemma~\ref{lem : basic prop of embmcg} \i{(ii)}. Thus, $z\notin P$.

    Thus, $z\in P_1\cup P_2$. Adjust notation so that $z\in P_1$ (see Figure~\ref{subfig: z in p_1}). Let $Q:=aPyP_yzP_1b_1$ and $Q':=zP_1xP_2b_2$ (see Figure~\ref{subfig: ears Q, Q' and xy}). Observe that the \cmcs{} $H\cup (P+P_1)\cup P_2\cup P_y$ can also be obtained from $H$ by adding the ears $Q$, $Q'$ and $xy$ in that order. In particular, $H\cup Q\cup Q'$ is a \cmcs{} of $G$ that is not induced. This contradicts Lemma~~\ref{lem : cmcs is induced}. Thus, $|\partial^B(H)|=1$. 
    \end{proof}

    Let $f$ denote the unique edge of $\partial^B(H)$. Note that $\partial(H)=\{e,f\}$ is indeed a balanced $2$-cut. It remains to prove the following claim.
    
    \begin{stat}
        $f$ is a \tedge{}.
        \label{stat : f is a 3edge}
    \end{stat}
    \begin{proof}
    We let $f:=uv$ where $v\in V(H)\cap B$. As $H$ is \mc{}, $d_H(v)\geqslant2$; thus $d_G(v)\geqslant3$. Suppose to the contrary that $d_G(u)=2$. By Lemma~\ref{lem : basic prop of embmcg} \i{(i)}, the neighbour of $u$ distinct from $v$, say $w$, has degree two in $G$. Let $P$ be an ear of $H$ containing $f$. Observe that $uw,e\in P$. By Lemma~\ref{lem : basic prop of embmcg}~\i{(ii)}, $uw$ is the only $2$-edge of $P$. Furthermore, by Lemma~\ref{lem : basic prop of embmcg}~\i{(i)}, each internal vertex of $P$, distinct from $u$ and $w$, has degree at least three in $G$. 

    \begin{figure}[htb]
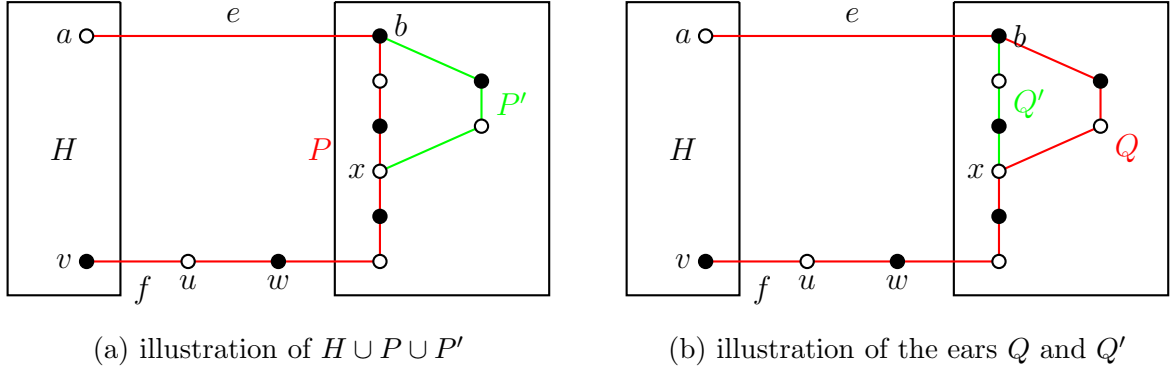

        \centering
        \begin{subfigure}{0.49\linewidth}
            \centering
            \begin{drawing}{1.5}
                \draw (-1,1.3)--(-2,1.3)--(-2,-1.3)--(-1,-1.3)--(-1,1.3);
            \draw (0.9,1.3)--(2.8,1.3)--(2.8,-1.3)--(0.9,-1.3)--(0.9,1.3);
            
            \draw [color=red](-1.3,1)--(1.3,1)--(1.3,-1)--(-1.3,-1);
            \draw [color=red] (1,0)node[left,nodelabel]{$P$};
            \draw [color=green](1.3,1)--(2.2,0.6)--(2.2,0.2)--(1.3,-0.2);
            \draw [color=green] (2.2,0.4)node[right,nodelabel]{$P'$};
            
            \draw (1.3,-0.2)node[left,nodelabel]{$x$};
            \draw (2.2,0.6)node[bnode]{};
            \draw (2.2,0.2)node[wnode]{};
            
            \draw (1.3,0.6)node[wnode]{};
            \draw (1.3,-0.6)node[bnode]{};
            \draw (1.3,0.2)node[bnode]{};
            \draw (1.3,-0.2)node[wnode]{};  
            
            \draw (-1.3,1)node[wnode]{}node[left,nodelabel]{$a$};
            \draw (1.3,1)node[bnode]{};
            \draw (1.3,1.1)node[right,nodelabel]{$b$};
            \draw (0,1)node[above,nodelabel]{$e$};
            \draw (-0.8,-1)node[below,nodelabel]{$f$};
            \draw (-1.3,-1)node[bnode]{}node[left,nodelabel]{$v$};

            \draw (1.3,-1)node[wnode]{};
            \draw (-0.4,-1)node[wnode]{}node[below,nodelabel]{$u$};
            \draw (0.4,-1)node[bnode]{}node[below,nodelabel]{$w$};
            
            \draw (-1.5,0)node[nodelabel]{$H$};
            \end{drawing}
            \caption{illustration of $H \cup P\cup P'$}
            \label{subfig: h cup p cup p'}    
        \end{subfigure}
        \begin{subfigure}{0.49\linewidth}
            \centering
            \begin{drawing}{1.5}
            \draw (-1,1.3)--(-2,1.3)--(-2,-1.3)--(-1,-1.3)--(-1,1.3);
            \draw (0.9,1.3)--(2.8,1.3)--(2.8,-1.3)--(0.9,-1.3)--(0.9,1.3);

            \draw [color=red](-1.3,1)--(1.3,1)--(2.2,0.6)--(2.2,0.2)--(1.3,-0.2)--(1.3,-1)--(-1.3,-1);
            \draw [color=green](1.3,1)--(1.3,-0.2);
            \draw [color=green] (1.3,0.4)node[right,nodelabel]{$Q'$};
            \draw [color=red] (2.2,0)node[right,nodelabel]{$Q$};
            
            \draw (1.3,-0.2)node[left,nodelabel]{$x$};
            \draw (2.2,0.6)node[bnode]{};
            \draw (2.2,0.2)node[wnode]{};
            
            \draw (1.3,0.6)node[wnode]{};
            \draw (1.3,-0.6)node[bnode]{};
            \draw (1.3,0.2)node[bnode]{};
            \draw (1.3,-0.2)node[wnode]{}; 
            
            \draw (-1.3,1)node[wnode]{}node[left,nodelabel]{$a$};
            \draw (1.3,1)node[bnode]{}node[right,nodelabel]{$b$};
            \draw (0,1)node[above,nodelabel]{$e$};
            \draw (-0.8,-1)node[below,nodelabel]{$f$};
            \draw (-1.3,-1)node[bnode]{}node[left,nodelabel]{$v$};

            \draw (1.3,-1)node[wnode]{};
            \draw (-0.4,-1)node[wnode]{}node[below,nodelabel]{$u$};
            \draw (0.4,-1)node[bnode]{}node[below,nodelabel]{$w$};
            
            \draw (-1.5,0)node[nodelabel]{$H$};
            \end{drawing}
            \caption{illustration of the ears $Q$ and $Q'$}
            \label{subfig: ears Q and Q'}    
        \end{subfigure}
        
        \caption{illustrations for the proof of statement~\ref{stat : f is a 3edge}}
        \label{fig:enter-label}
    \end{figure}

    Since $d_G(b)\geqslant3$, there is an ear $P'$ of $H\cup P$ starting from $b$ (see Figure~\ref{subfig: h cup p cup p'}). Let $x\in A$ be the other end of~$P'$. Observe that $x$ is an internal vertex of $P$ that is distinct from $u$. Let $Q:=abP'xPv$ and $Q':=bPx$ (see Figure~\ref{subfig: ears Q and Q'}). Observe that the \cmcs{} $H\cup P\cup P'$ can also be obtained from $H$ by adding the ears $Q$ and $Q'$ in that order. It follows from the preceding paragraph that each vertex of $Q'$ has degree at least three in $G$. Consequently, the ear $Q'$ has no $2$-edges. This contradicts Lemma~\ref{lem : basic prop of embmcg} \i{(ii)}. Thus, $d_G(u)\geqslant3$ and $f$ is indeed a \tedge{}.
    \end{proof}
    This completes the proof of Theorem~\ref{thm : 2-cut property}.
\end{proof}

Note that Theorem~\ref{thm : 2-cut property} does not hold for all \mbmcg{}s --- that is, it requires the extremality hypothesis. As we noted in Section 1, the graph $G$ shown in Figure~\ref{fig: counterexample mbmcg} is a \mbmcg{}. Observe that $e$ is a $3$-edge that is not contained in any $2$-cut since $G-e$ is \kC{2}. 

In the next section, we establish an induction tool based on balanced $2$-cuts. This tool, coupled with Theorem~\ref{thm : 2-cut property}, will be used in Section~\ref{sec : main theorem} to prove the Main Theorem (\ref{thm : evm}).

\section{An induction tool using balanced $2$-cuts}
\label{sec : 2-cut induction tool}


We begin this section by defining an operation that ``breaks" a given bipartite graph, that has a balanced $2$-cut, into ``smaller" bipartite graphs, and its converse operation.

Let $G[A,B]$ denote a bipartite graph that has a balanced $2$-cut, say $F:=\partial(X)$. Let $H_1:=G[X]$ and $H_2:=G[\overline{X}]$. As $F$ is balanced, for each $i\in \{1,2\}$, let $a_i\in A$ and $b_i\in B$ denote the ends of the edges of $F$ in $H_i$, as shown in Figure~\ref{subfig : G*}. Let $G_1$ be the graph obtained from $H_1$ by adding the ear $a_1v_1u_1b_1$ (of length three), as shown in Figure~\ref{subfig : G_1 and G_2}. The graph $G_2$ is obtained from $H_2$ analogously. We say that the (bipartite) graphs $G_1$ and~$G_2$ are obtained from $G$ by a \i{balanced $2$-cut decomposition}, or simply by a \i{$2$-cut decomposition},  across the $2$-cut $F$. Note that $u_1v_1$ is a $2$-edge in $G_1$; likewise, $u_2v_2$ is a $2$-edge in $G_2$.

Conversely, let $G_1$ and $G_2$ be disjoint bipartite graphs, each of which has a $2$-edge, say $u_1v_1$ and $u_2v_2$, respectively. Let $a_1$ denote the neighbour of $v_1$ that is distinct from $u_1$, and let $b_1$ denote the neighbour of $u_1$ that is distinct from $v_1$, as shown in Figure~\ref{subfig : G_1 and G_2}. The vertices $a_2$ and $b_2$ are defined analogously. Let $G$ be the graph obtained from $G_1\cup G_2$ by deleting the vertices $u_1,v_1,u_2$ and $v_2$, and adding the edges $a_1b_2$ and $a_2b_1$, as shown in Figure~\ref{subfig : G*}. We say that the (bipartite) graph $G$ is obtained from $G_1$ and $G_2$ by \i{$2$-edge splicing} across the pair of $2$-edges $\{u_1v_1,u_2v_2\}$. Note that $\{a_1b_2,a_2b_1\}$ is a balanced $2$-cut of $G$.

Observe that the $2$-cut decomposition operation is the ``inverse" of the $2$-edge splicing operation --- that is, $G_1$ and $G_2$ are obtained from $G$ by $2$-cut decomposition across a balanced $2$-cut if and only if $G$ is obtained from $G_1$ and $G_2$ by $2$-edge splicing across a pair of $2$-edges. We now show that these operations preserve the matching covered, minimality and extremality properties.

\begin{figure}[htb]
    \centering
        
    \begin{subfigure}{0.45\linewidth}
        \centering
        \begin{tikzpicture}[thick, scale=1.5,nodelabel/.style={rounded corners,fill=none,inner sep=5pt,draw=none}]
        
        \draw (-1,1)--(-2,1)--(-2,-1)--(-1,-1)--(-1,1);
        \draw (1,1)--(2,1)--(2,-1)--(1,-1)--(1,1);
        \draw [color=red](-1.3,0.7)--(1.3,0.7)(1.3,-0.7)--(-1.3,-0.7);
        
        \draw (-1.3,0.7)node[wnode]{}node[left,nodelabel]{$a_1$};
        \draw (1.3,0.7)node[bnode]{}node[right,nodelabel]{$b_2$};
        \draw (-1.3,-0.7)node[bnode]{}node[left,nodelabel]{$b_1$};
        \draw (1.3,-0.7)node[wnode]{}node[right,nodelabel]{$a_2$};
        
        \draw (-1.5,-1)node[below,nodelabel]{$H_1:=G[X]$};
        \draw (1.5,-1)node[below,nodelabel]{$H_2:=G[\overline{X}]$};
        
        \end{tikzpicture}
        \caption{$G$}
        \label{subfig : G*}
        \label{fig}
    \end{subfigure}
    \begin{subfigure}{0.45\linewidth}
        \centering
        \begin{tikzpicture}[thick, scale=1.5,nodelabel/.style={rounded corners,fill=none,inner sep=5pt,draw=none}]
        
        \draw (-1,1)--(-2,1)--(-2,-1)--(-1,-1)--(-1,1);
        \draw (1,1)--(2,1)--(2,-1)--(1,-1)--(1,1);
        \draw [color=green](1.3,0.7)--(0.7,0.2)--(0.7,-0.2)--(1.3,-0.7);
        \draw [color=green](-1.3,0.7)--(-0.7,0.2)--(-0.7,-0.2)--(-1.3,-0.7);
        
        \draw (0.7,0.2)node[wnode]{}node[above left,nodelabel]{$u_2$};
        \draw (0.7,-0.2)node[bnode]{}node[below left,nodelabel]{$v_2$};
        \draw (-0.7,0.2)node[bnode]{}node[above right,nodelabel]{$v_1$};
        \draw (-0.7,-0.2)node[wnode]{}node[below right,nodelabel]{$u_1$};
        
        \draw (-1.3,0.7)node[wnode]{}node[left,nodelabel]{$a_1$};
        \draw (1.3,0.7)node[bnode]{}node[right,nodelabel]{$b_2$};
        \draw (-1.3,-0.7)node[bnode]{}node[left,nodelabel]{$b_1$};
        \draw (1.3,-0.7)node[wnode]{}node[right,nodelabel]{$a_2$};
        
        \draw (-1.5,-1)node[below,nodelabel]{$G_1$};
        \draw (1.5,-1)node[below,nodelabel]{$G_2$};
        
        \end{tikzpicture}
        \caption{$G_1$ and $G_2$}
        \label{subfig : G_1 and G_2}
        \label{fig}
    \end{subfigure}
    
    \caption{illustration of balanced $2$-cut decomposition and $2$-edge splicing}
    \label{fig : 2-cut decomp and 2-edge splicing}
\end{figure}
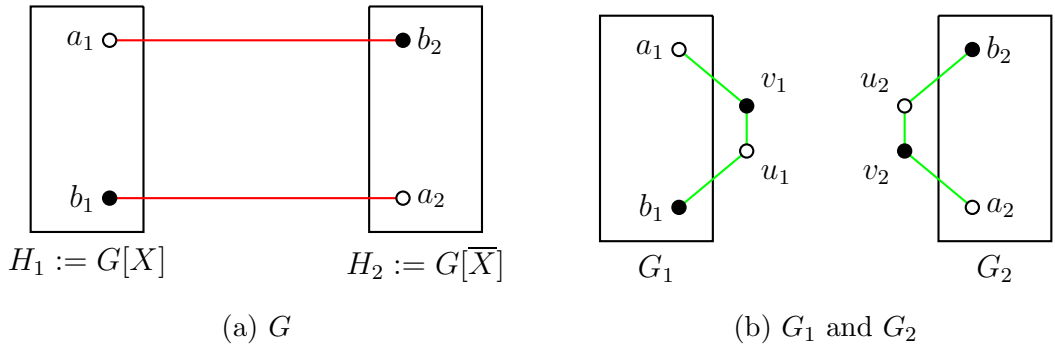

\begin{thm} {\sc [Balanced $2$-cut Induction Tool]}\\
    Let $G$ be a bipartite graph that has a balanced $2$-cut $F$, and let $G_1$ and $G_2$ be the (bipartite) graphs obtained from $G$ by a $2$-cut decomposition across $F$. Then the following statements hold:
    \begin{enumerate}
        \item $G$ is \mc{} if and only if $G_1$ and $G_2$ are both \mc{};
        \item furthermore, $G$ is minimal if and only if $G_1$ and $G_2$ are both minimal; and
        \item finally, $G$ is extremal if and only if $G_1$ and $G_2$ are both extremal.
    \end{enumerate}
    \label{thm : 2-cut induction tool}
\end{thm}
\begin{proof}
    We adopt all of the notations from the above definition of $2$-cut decomposition, as shown in Figure~\ref{fig : 2-cut decomp and 2-edge splicing}. We will prove the three statements one by one. 
    \begin{stat}
        $G$ is \mc{} if and only if $G_1$ and $G_2$ are both \mc{}.
        \label{stat : mc}
    \end{stat}
    \begin{proof}
        First suppose that $G$ is \mc{}. Let $C$ denote a conformal cycle containing $a_1b_2$. Note that $a_2b_1\in C$ as well. Let $(C,P_1,\dots,P_r)$ be an ear sequence of $G$; its existence is guaranteed by Theorem~\ref{thm : strong ear decomp}. Observe that each ear $P_i$ is a subgraph of exactly one of $H_1$~and~$H_2$. Let $P_{11},P_{12},\dots ,P_{1r_1}$ denote those ears that are subgraphs of $H_1$ and appear in that order in the ear sequence $(C,P_1,\dots ,P_r)$. Let $C_1:=(C\cap H_1)+a_1v_1u_1b_1$. Observe that $(C_1,P_{11},P_{12},\dots ,P_{1r_1})$ is an ear sequence of $G_1$. Thus, by the Ear Decomposition Theorem~(\ref{thm : odd ear decomp}), the bipartite graph $G_1$ is \mc{}. An analogous argument proves that $G_2$ is \mc{}.

        Conversely, suppose that $G_1$ and $G_2$ are \mc{}. Let $C_i$ be a conformal cycle of $G_i$ containing the $2$-edge $u_iv_i$ for $i\in\{1,2\}$. Let $(C_1,P_{11},P_{12},\dots ,P_{1r_1})$ and $(C_2,P_{21},P_{22},\dots ,P_{2r_2})$ be ear sequences of $G_1$ and $G_2$, respectively; these exist due to Theorem~\ref{thm : strong ear decomp}. Let $C:=(C_1-u_1-v_1)+ (C_2-u_2-v_2)+a_1b_2+a_2b_1$. Observe that $(C,P_{11},P_{12},\dots P_{1r_1},P_{21},P_{22},\dots ,P_{2r_2})$ is an ear sequence of $G$. Thus, by the Ear Decomposition Theorem~(\ref{thm : odd ear decomp}), $G$ is \mc{}. 
    \end{proof}

    Henceforth, we assume that $G$ is \mc{}, or equivalently that $G_1$ and $G_2$ are both \mc{}.
    
    \begin{stat}
        $G$ is minimal if and only if $G_1$ and $G_2$ are both minimal.
        \label{stat : minimal}
    \end{stat}
    \begin{proof}
        First suppose that $G$ is minimal. We claim that $G_1$ is also minimal. Suppose not, and let $e$ be a removable edge of $G_1$. Note that $e\in H_1$ since each of the remaining three edges has an end of degree two. Observe that $G_1-e$ and $G_2$ are obtained by a $2$-cut decomposition of $G-e$. Consequently, by~\ref{stat : mc}, the graph $G-e$ is matching covered; this contradicts the minimality of $G$. Thus, $G_1$ is minimal. Likewise, $G_2$ is minimal.
    
        Conversely, suppose that $G_1$ and $G_2$ are both minimal. We claim that $G$ is minimal. Suppose not, and let $e$ be a removable edge of $G$. Note that, since $F$ is $2$-cut, $e$ belongs to exactly one of $H_1$~and~$H_2$. Adjust notation so that $e\in H_1$. As in the preceding paragraph, $G_1-e$ and $G_2$ are obtained by a $2$-cut decomposition of $G-e$. Thus, by~\ref{stat : mc}, the graph $G_1-e$ is matching covered; this contradicts the minimality of $G_1$. Hence, $G$ is minimal.    
    \end{proof}

    Henceforth, we assume that $G$ is a \mbmcg{}, or equivalently that $G_1$ and $G_2$ are both \mbmcg{}s.

    \begin{stat}
        $G$ is extremal if and only if $G_1$ and $G_2$ are both extremal.
        \label{stat : extremal}
    \end{stat}

    \begin{proof}
        In order to argue extremality, we find it convenient to use $\lambda :=|V_2(G)|-2(m-n+2)$ to denote the ``surplus" vertices of degree two in $G$. Analogously, we define $\lambda_1$ and $\lambda_2$ for the graphs $G_1$ and $G_2$, respectively. By Corollary~\ref{cor : bound evm}, since all three graphs are \mbmcg{}s, the quantities $\lambda$, $\lambda_1$ and $\lambda_2$ are non-negative. Note that $G$ is extremal if and only if $\lambda=0$. Likewise, for $i\in \{1,2\}$, the graph $G_i$ is extremal if and only if $\lambda_i=0$. Thus, it suffices to prove that $\lambda=0$ if and only if $ \lambda_1=0$ and $\lambda_2=0$.
    
        For $i\in \{1,2\}$, we let $n_i:=|V(G_i)|$ and $m_i:=|E(G_i)|$. The reader may find it useful to look at Figure~\ref{fig : 2-cut decomp and 2-edge splicing}. Note that $n_1+n_2=n+4$ and $m_1+m_2=m+4$. Observe that each vertex $w$ of $G$ corresponds to a vertex of $G_1$ or of $G_2$ (that is distinct from $u_1,v_1,u_2$ and $v_2$) whose degree is the same as that of $w$. This implies that $|V_2(G_1)|+|V_2(G_2)|=|V_2(G)|+4$. Using these three equalities, and the definitions of $\lambda,\lambda_1$ and $\lambda_2$:
        \begin{align*}
            \lambda_1+\lambda_2=&|V_2(G_1)|+|V_2(G_2)|-2(m_1+m_2-n_1-n_2+4)\\
            =&|V_2(G)|+4-2(m-n+4)\\
            =&|V_2(G)|-2(m-n+2)=\lambda
        \end{align*}
        Thus, since $\lambda,\lambda_1,\lambda_2$ are non-negative, $\lambda=0$ if and only if $ \lambda_1=0$ as well as $\lambda_2=0$. 
    \end{proof}
    This completes the proof of Theorem~\ref{thm : 2-cut induction tool}.
\end{proof}

\section{Main Theorem: Characterization of $\c{H}_2$}
\label{sec : main theorem}

In this section, we prove our Main Theorem (\ref{thm : evm}) that provides a characterization of \embmcg{}s. Before doing so, we prove a lemma that will be used in the base case of our inductive proof.

\begin{lem}
    Every \embmcg{}, that has precisely two vertices of degree at least three, is obtained from a star by \isojoin{}. 
    \label{lem : base case - main theorem}
\end{lem}
\begin{proof}
    Let $G$ be an \embmcg{} with precisely two vertices of degree at least three, say $a$ and $b$. Let $(C,P_1,\dots, P_r)$ be an ear sequence. Clearly, $G$ is not a cycle; thus $r\geqslant 1$. Also, each end of each ear $P_i$ has degree at least three in $G$. Thus, every ear is an $ab$-path. Consequently, $a,b\in C$. By Lemma~\ref{lem : basic prop of embmcg} \i{(ii)}, each $P_i$ has length three. Furthermore, by Lemma~\ref{lem : basic prop of embmcg} \i{(iii)}, the vertices $a$ and $b$ split $C$ into two paths --- each of length three. In particular, $C$ is a $6$-cycle. Observe that $G$ is indeed obtained from the star $K_{1,r+2}$ by \isojoin{}.
\end{proof}

We will now use the balanced $2$-cut property (Theorem~\ref{thm : 2-cut property}) and the balanced $2$-cut induction tool (Theorem~\ref{thm : 2-cut induction tool}), along with the above lemma, to prove our characterization of the class of \embmcg{}s. 

\begin{restate}{\ref{thm : evm}}{\sc[Main Theorem: Characterization of $\c{H}_2$]}\\
    A graph $G$ belongs to $\c{H}_2$  if and only if it is obtained from a \htree{} by \isojoin{}.
\end{restate}

\begin{proof}
    The reverse implication is simply the main statement of Theorem~\ref{thm : embmcg TcupT'} that we have already proved.
    
    For the forward direction, we let $G[A,B]$ denote a member of $\c{H}_2$, and we proceed by \mbox{induction} on the order of $G$. First suppose that $|V_3|\leqslant 2$. Since $|A|=|B|$, it follows that $|V_3|\in\{0,2\}$. If $|V_3|=0$, then $G$ is a cycle and $|V|=|V_2|=2(m-n+2)=4$; thus $G$ is~$C_4$. On the other hand, if $|V_3|=2$, we invoke Lemma~\ref{lem : base case - main theorem}. In either case, we conclude that $G$ is obtained from a \htree{} by \isojoin{}.    

    Now, suppose that $|V_3|\geqslant3$. By Corollary~\ref{cor : counting prop of embmcg} \i{(iii)}, we note that $|E_{3}|=|V_3|-2\geqslant 1$, and we let $a_1b_2$ denote a \tedge{}. By Theorem~\ref{thm : 2-cut property}, there exists another \tedge{} $b_1a_2$ such that $\{a_1b_2,b_1a_2\}$ is a balanced $2$-cut, say $\partial(X)$, where $\{a_1,b_1\}\subseteq X\subset V(G)$. Now, let $G_1$ and $G_2$ be the graphs obtained by $2$-cut decomposition across $\partial(X)=\{a_1b_2,b_1a_2\}$. In particular, $G_1$ is obtained from $G[X]$ by adding the ear $a_1v_1u_1b_1$; likewise, $G_2$ is obtained from $G[\overline{X}]$ by adding the ear $a_2v_2u_2b_2$. See Figure~\ref{fig : 2-cut decomp and 2-edge splicing} for an illustration.
    
    By Theorem~\ref{thm : 2-cut induction tool}, each of $G_1$ and $G_2$ belongs to $\c{H}_2$. Since each of $a_1,b_1,a_2$ and $b_2$ has degree at least three in $G$, it follows that each of $X$ and $\overline{X}$ has at least four vertices. This implies that $6\leqslant|V(G_i)|<|V(G)|$ for each $i\in \{1,2\}$. Thus, by the induction hypothesis, each of \mbox{$G_1$ and $G_2$} is obtained from a \htree{} by \isojoin{}; also, neither of them is isomorphic to $C_4$. By Proposition~\ref{prop : t cup t'}, for each $i\in \{1,2\}$, the graph $G_i-E_2(G_i)$ has precisely two components that are isomorphic \htree{}s, say $T_i$ and $T_i'$. We let $\phi_i:V(T_i)\mapsto V(T_i')$ denote the corresponding isomorphism as per the \isojoin{} used to obtain~$G_i$. 

    \begin{figure}[htb]
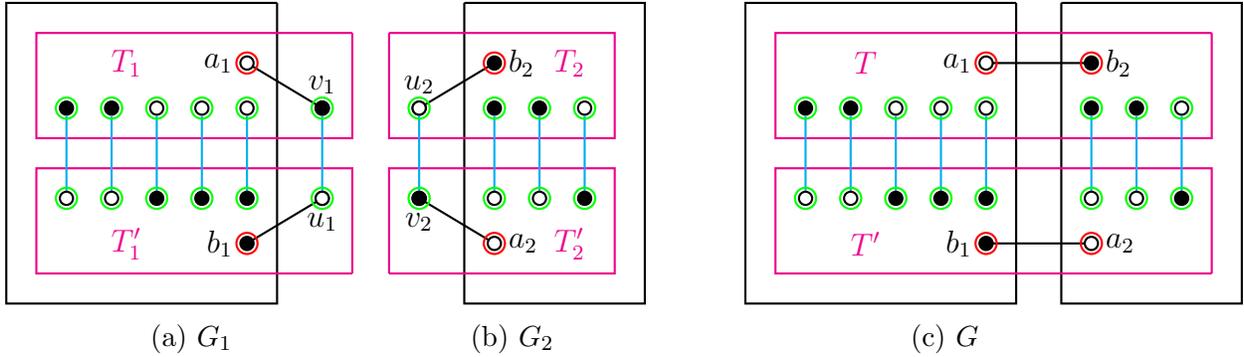

        \centering
        \begin{subfigure}{0.3\linewidth}
            \begin{drawing}{2}
                \let\edgecolor\cyanedge
                \draw (-1,1)--(-2.8,1)--(-2.8,-1)--(-1,-1)--(-1,1);
                \draw[color=magenta] (-0.5,0.8)--(-2.6,0.8)--(-2.6,0.1)--(-0.5,0.1)--(-0.5,0.8);
                \draw[magenta] (-0.5,-0.8)--(-2.6,-0.8)--(-2.6,-0.1)--(-0.5,-0.1)--(-0.5,-0.8);
                \draw (-1.2,0.6)--(-0.7,0.3)(-0.7,-0.3)--(-1.2,-0.6);

                \draw[color=magenta] (-2,0.6)node[nodelabel]{$T_1$};
                \draw[color=magenta] (-2,-0.6)node[nodelabel]{$T'_1$};
                \draw (-0.7,0.3)node[bnode]{}node[above,nodelabel]{$v_1$};
                \draw (-0.7,-0.3)node[wnode]{}node[below,nodelabel]{$u_1$};

                \foreach \x in {-1.2,-1.5,-1.8}{
                    \draw (\x,0.3)node[wnode]{}edge(\x,-0.3)node[bnode]{};
                    \draw[color=green] (\x,0.3) circle (2pt);
                    \draw[color=green] (\x,-0.3) circle (2pt);
                }
                \foreach \x in {-2.1,-2.4,-0.7}{
                    \draw (\x,0.3)node[bnode]{}edge(\x,-0.3)node[wnode]{};
                    \draw[color=green] (\x,0.3) circle (2pt);
                    \draw[color=green] (\x,-0.3) circle (2pt);
                }
                
                \draw (-1.2,0.6)node[wnode]{}node[left,nodelabel]{$a_1$};
                \draw (-1.2,-0.6)node[bnode]{}node[left,nodelabel]{$b_1$};
                \draw[color=red] (-1.2,0.6) circle (2pt);
                \draw[color=red] (-1.2,-0.6) circle (2pt);

            \end{drawing}
            \caption{$G_1$}
            \label{subfig : G_1}
        \end{subfigure}
        \begin{subfigure}{0.2\linewidth}
            \centering
            \begin{drawing}{2}
                \let\edgecolor\cyanedge
                \draw (1,1)--(2.2,1)--(2.2,-1)--(1,-1)--(1,1);
                \draw[color=magenta] (0.5,0.8)--(2,0.8)--(2,0.1)--(0.5,0.1)--(0.5,0.8);
                \draw[color=magenta] (0.5,-0.8)--(2,-0.8)--(2,-0.1)--(0.5,-0.1)--(0.5,-0.8);
                \draw (1.2,0.6)--(0.7,0.3)(0.7,-0.3)--(1.2,-0.6);

                \draw[color=magenta] (1.7,0.6)node[nodelabel]{$T_2$};
                \draw[color=magenta] (1.7,-0.6)node[nodelabel]{$T'_2$};
                
                \draw (0.7,0.3)node[wnode]{}node[above,nodelabel]{$u_2$};
                \draw (0.7,-0.3)node[bnode]{}node[below,nodelabel]{$v_2$};

                \foreach \x in {1.2,1.5}{
                    \draw (\x,0.3)node[bnode]{}edge(\x,-0.3)node[wnode]{};
                    \draw[color=green] (\x,0.3) circle (2pt);
                    \draw[color=green] (\x,-0.3) circle (2pt);
                }
                \foreach \x in {1.8,0.7}{
                    \draw (\x,0.3)node[wnode]{}edge(\x,-0.3)node[bnode]{};
                    \draw[color=green] (\x,0.3) circle (2pt);
                    \draw[color=green] (\x,-0.3) circle (2pt);
                }
                
                \draw (1.2,0.6)node[bnode]{}node[right,nodelabel]{$b_2$};
                \draw (1.2,-0.6)node[wnode]{}node[right,nodelabel]{$a_2$};
                \draw[color=red] (1.2,0.6) circle (2pt);
                \draw[color=red] (1.2,-0.6) circle (2pt);

            \end{drawing}
            \caption{$G_2$}
            \label{subfig : G_2}
        \end{subfigure}
        \begin{subfigure}{0.48\linewidth}
            \hfill
            \begin{drawing}{2}
                \let\edgecolor\cyanedge
                \draw (-1,1)--(-2.8,1)--(-2.8,-1)--(-1,-1)--(-1,1);
                \draw (-0.7,1)--(0.5,1)--(0.5,-1)--(-0.7,-1)--(-0.7,1);
                \draw[color=magenta] (0.3,0.8)--(-2.6,0.8)--(-2.6,0.1)--(0.3,0.1)--(0.3,0.8);
                \draw[magenta] (0.3,-0.8)--(-2.6,-0.8)--(-2.6,-0.1)--(0.3,-0.1)--(0.3,-0.8);
                \draw (-1.2,0.6)--(-0.5,0.6)(-0.5,-0.6)--(-1.2,-0.6);

                \draw[color=magenta] (-2,0.6)node[nodelabel]{$T$};
                \draw[color=magenta] (-2,-0.6)node[nodelabel]{$T'$};

                \foreach \x in {-1.2,-1.5,-1.8}{
                    \draw (\x,0.3)node[wnode]{}edge(\x,-0.3)node[bnode]{};
                    \draw[color=green] (\x,0.3) circle (2pt);
                    \draw[color=green] (\x,-0.3) circle (2pt);
                }
                \foreach \x in {-2.1,-2.4}{
                    \draw (\x,0.3)node[bnode]{}edge(\x,-0.3)node[wnode]{};
                    \draw[color=green] (\x,0.3) circle (2pt);
                    \draw[color=green] (\x,-0.3) circle (2pt);
                }
                \foreach \x in {-0.5,-0.2}{
                    \draw (\x,0.3)node[bnode]{}edge(\x,-0.3)node[wnode]{};
                    \draw[color=green] (\x,0.3) circle (2pt);
                    \draw[color=green] (\x,-0.3) circle (2pt);
                }
                \foreach \x in {0.1}{
                    \draw (\x,0.3)node[wnode]{}edge(\x,-0.3)node[bnode]{};
                    \draw[color=green] (\x,0.3) circle (2pt);
                    \draw[color=green] (\x,-0.3) circle (2pt);
                }
                
                \draw (-1.2,0.6)node[wnode]{}node[left,nodelabel]{$a_1$};
                \draw (-1.2,-0.6)node[bnode]{}node[left,nodelabel]{$b_1$};
                \draw[color=red] (-1.2,0.6) circle (2pt);
                \draw[color=red] (-1.2,-0.6) circle (2pt);
                \draw (-0.5,0.6)node[bnode]{}node[right,nodelabel]{$b_2$};
                \draw (-0.5,-0.6)node[wnode]{}node[right,nodelabel]{$a_2$};
                \draw[color=red] (-0.5,0.6) circle (2pt);
                \draw[color=red] (-0.5,-0.6) circle (2pt);

            \end{drawing}
            \caption{$G$}
            \label{subfig : G}
        \end{subfigure}
        \caption{illustration for the proof of the Main Theorem}
        \label{fig : main theorem}
    \end{figure}
    
    
    Since $u_1v_1$ is a $2$-edge of $G_1$, we may adjust notation so that $v_1,a_1\in T_1$ and $u_1,b_1\in T_1'$. See Figure~\ref{fig : main theorem}. Likewise, we adjust notation so that $u_2,b_2\in T_2$ and $v_2,a_2\in T_2'$. Note that $\phi_1(v_1)=u_1$. Consequently, $\phi_1(a_1)=b_1$. Likewise, $\phi_2(b_2)=a_2$. We let $T:=(T_1-v_1)+(T_2-u_2)+a_1b_2$ and $T':=(T'_1-u_1)+(T'_2-v_2)+a_2b_1$. Note that $T$ and $T'$ are \htree{}s. Observe that $\phi:V(T)\mapsto V(T')$, as defined below, is an isomorphism between $T$ and $T'$.
    
    \[\phi(v):=\begin{cases}
        \phi_1(v) & \t{if }v\in V(T_1)\\
        \phi_2(v) & \t{if }v\in V(T_2)\\
    \end{cases}\]
    
    Finally, note that $G$ is obtained from $T$ and $T'$ by \isojoin{} as per the above defined isomorphism. This completes the proof of the Main Theorem.
\end{proof}

\section{Other notions of extremality}
\label{sec : other classes}
We now shift our focus to the remaining notions of extremality --- that is, the extremal classes $\c{H}_3$, $\c{H}_4$, $\c{H}_0$ and $\c{H}_1$ --- in that order.  

\subsection{Characterization of $\c{H}_3$}
\label{subsec : h3}

We begin by deducing the lower bound on the number of degree two vertices in a \mbmcg{} solely in terms of $n$, that was stated in Subsection~\ref{subsec : bounds}, from the lower bound of Lov\'asz and Plummer (Corollary~\ref{cor : bound evm}). 

\begin{cor}
    In a \mbmcg{}, $|V_2|\geqslant\frac{n}{2}+2$.
    \label{cor : bound evn}
\end{cor}
\begin{proof}
    Using the fact that $n=|V_2|+|V_3|$, and the handshaking lemma: 
    \[m-n=\Bigg(\sum_{v\in V_3}\frac{d(v)}{2}+|V_2|\Bigg)-n\geqslant \Bigg(\frac{3}{2}|V_3|+|V_2|\Bigg)-n=\frac{1}{2}|V_3|\]
    We invoke Corollary~\ref{cor : bound evm} to infer that $|V_2|\geqslant 2(m-n+2)\geqslant |V_3|+4=n-|V_2|+4$. Consequently, $|V_2|\geqslant \frac{n}{2}+2$.
\end{proof}

We are now ready to prove our characterization for $\c{H}_3$. To this end, we shall enforce equality in the proof of the above corollary.

\begin{restate}{\ref{thm : evn}}{\sc[Characterization of $\c{H}_3$]}\\
    A graph $G$ belongs to $\c{H}_3$ if and only if it is obtained from a cubic \htree{} by \isojoin{}.
\end{restate}
\begin{proof}
    The reverse implication is simply statement \i{(i)} of Theorem~\ref{thm : embmcg TcupT'} that we have already proved. It remains to prove the forward implication. 
    
    Let $G$ be a member of $\c{H}_3$. That is, $G$ is a \mbmcg{} such that $|V_2|=\frac{n}{2}+2$. Thus, $G$ satisfies each inequality in the proof of Corollary~\ref{cor : bound evn} with equality. Firstly, $|V_2|=2(m-n+2)$; thus $G\in \c{H}_2$. By the Main Theorem (\ref{thm : evm}), $G$ is obtained from a \htree{} $T$ by \isojoin{}. Secondly, each vertex in $V_3$ has degree precisely three in $G$. This implies that each non-leaf of $T$ is cubic. Thus, $T$ is a cubic \htree{}. 
\end{proof}

We end this section by noting the following proper containment that is established within the above proof.

\begin{cor}
    $\c{H}_3\subset \c{H}_2$.\qed
    \label{cor : h3 subset h2}
\end{cor}

\subsection{Characterization of $\c{H}_4$}
\label{subsec : h4}

We will follow the same approach as in the previous subsection. We begin by deducing the upper bound $|E|\leqslant \frac{3n-6}{2}$ (see Lov\'asz and Plummer \cite[Theorem 4.2.3]{lopl86}), that was stated in Subsection~\ref{subsec : bounds}, from the lower bound of Lov\'asz and Plummer (Corollary~\ref{cor : bound evm}).

\begin{cor}
    A \mbmcg{}, distinct from $C_4$, has at most $\frac{3n-6}{2}$ edges.
    \label{cor : bound ee}
\end{cor}
\begin{proof}
    Let $G[A,B]$ denote a \mbmcg{} distinct from $C_4$. Note that $n\geqslant6$ if and only if $n\leqslant \frac{3n-6}{2}$. Consequently, if $G$ is a cycle, $m=n\leqslant \frac{3n-6}{2}$. 
    
    Now suppose that $G$ is not a cycle; in other words, $|V_3|\geqslant 1$. Since $|A|=|B|$, we infer $|V_3|\geqslant 2$. Using Corollary~\ref{cor : bound evm} and the fact that $n=|V_2|+|V_3|$, we have $2(m-n+2)\leqslant |V_2|=n-|V_3|$. Since $|V_3|\geqslant 2$, we get $2(m-n+2)\leqslant n-2$. On simplifying, we conclude $m\leqslant \frac{3n-6}{2}$.
\end{proof}

We shall now enforce equality in the proof of the above corollary to prove our characterization of $\c{H}_4$.

\begin{restate}{\ref{thm : ee}}{\sc[Characterization of $\c{H}_4$]}\\
    A graph $G$ belongs to $\c{H}_4$ if and only if it is obtained from a star by \isojoin{}.
\end{restate}
\begin{proof}
    The reverse implication is settled by statement \i{(ii)} of Theorem~\ref{thm : embmcg TcupT'} for stars $K_{1,p}$ where $p\geqslant3$, and the paragraph above it for $K_{1,2}$. It remains to prove the forward implication.
    
    Let $G$ be a member of $\c{H}_4$. That is, $G$ is a \mbmcg{} such that $|E|=\frac{3n-6}{2}$. If $G$ is a cycle, then $n=6$ and $G$ is $C_6$ which is obtained from the star $K_{1,2}$ by \isojoin{}. 
    
    Now, suppose that $G$ is not a cycle. Consequently, $G$ satisfies all of the inequalities in the last paragraph of the proof of Corollary~\ref{cor : bound ee} with equality. In particular, $|V_2|=2(m-n+2)$ and $|V_3|=2$. In other words, $G$ is a member of $\c{H}_2$ that has precisely two vertices of degree at least three. Thus, by Lemma~\ref{lem : base case - main theorem}, we conclude that $G$ is obtained from a star by \isojoin{}. 
\end{proof}

We end this subsection with a couple of easy consequences. The first of them is the following proper containment that is already established within the above proof.

\begin{cor}
    $\c{H}_4-C_6\subset \c{H}_2$.\qed
    \label{cor : h4 subset h2}
\end{cor}

Note that $\Theta$, shown in Figure~\ref{fig: theta}, belongs to $\c{H}_3$ as well as $\c{H}_4$. Conversely, let $G$ be any graph in $\c{H}_3\cap \c{H}_4$. Since $G\in \c{H}_4$, by Theorem~\ref{thm : ee}, $G$ is obtained from a star $T:=K_{1,p}$ by \isojoin{}. On the other hand, since $G\in \c{H}_3$, by Theorem~\ref{thm : evn}, $T$ is a cubic \htree{}. Thus $p=3$; in other words, $G$ is $\Theta$. This proves the following.

\begin{cor}
    $\c{H}_4\cap \c{H}_3=\{\Theta\}$. \qed
    \label{cor : h4 cap h3 = theta}
\end{cor}

\subsection{Characterization of $\c{H}_0$}
\label{subsec : h0}

Recall the bicontraction and the \bicontract{} operations as defined in Subsection~\ref{subsec : h0 and h1}. In this subsection, we will first show that \bicontract{} preserves \eem{}ity. Then, we will use this to deduce our characterization of $\c{H}_0$. 

In order to achieve the above, we first define the inverse of \bicontract{} as well as of bicontraction. Observe that the bicontraction operation (when applicable) results in a vertex of degree at least two, whereas the \bicontract{} operation (when applicable) results in a vertex of degree at least four.  

Let $G'$ be a graph with a vertex $v'$ of degree two or more. We partition the edges of~$\partial(v')$ into two sets $F_1$ and $F_2$ each of size at least one. Let $G$ be obtained from $G'-v'$ by (i)~introducing three new vertices $v,v_1$ and $v_2$, (ii)~adding two edges $vv_1$ and $vv_2$, and (iii)~for each $i\in \{1,2\}$, adding an edge joining $u$ and $v_i$ for each $uv\in F_i$. See Figure~\ref{fig: bicontraction and bisplitting} for an illustration. We say that $G$ is obtained from $G'$ by \i{bisplitting} the vertex $v'$. Observe that, in~$G$, the vertex $v$ has degree two and each of its neighbours $v_1$ and $v_2$ has degree at least two.

Furthermore, if $v'$ has degree at least four and each of the sets $F_1$ and $F_2$ has size at least two, we say that $G$ is obtained from $G'$ by \i{\bisplit{}} of the vertex $v'$. In this case, note that the vertex $v$ has degree two and each of its neighbours $v_1$ and $v_2$ has degree at least three. See Figure~\ref{fig: restricted bicontraction and restricted bisplitting} for an illustration.

\begin{figure}[htb]
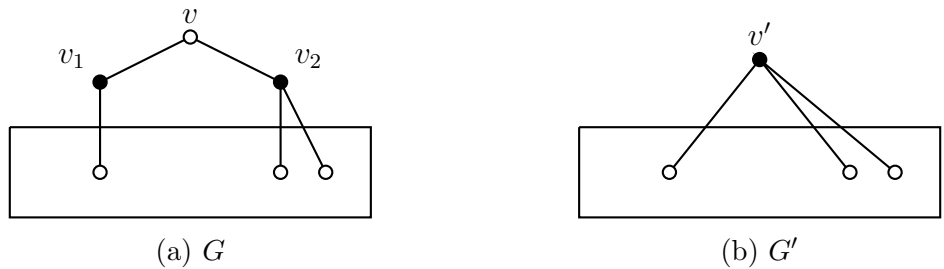

    \centering  
    \begin{subfigure}{0.45\linewidth}
        \centering
        \begin{drawing}{1.2}
    
        \draw (-2,1.5) -- (2,1.5) -- (2,0.5) -- (-2,0.5) -- (-2,1.5);
        \draw (-1,1) -- (-1,2);
        \draw (-1,2) -- (0,2.5);
        \draw (-1,1)node[wnode]{};
        \draw (-1,2)node[bnode]{}node[above left,nodelabel]{$v_1$};
    
        \draw (1,1) -- (1,2) -- (1.5,1);
        \draw (1,2) -- (0,2.5);
        \draw (1,1)node[wnode]{};
        \draw (1.5,1)node[wnode]{};
        \draw (1,2)node[bnode]{}node[above right,nodelabel]{$v_2$};
    
        \draw (0,2.5)node[wnode]{}node[above,nodelabel]{$v$};

        \end{drawing}
        \caption{$G$}
        \label{fig:my_label}
    \end{subfigure}
    \begin{subfigure}{0.45\linewidth}
        \centering
        \begin{drawing}{1.2}

        \draw (-2,1.5) -- (2,1.5) -- (2,0.5) -- (-2,0.5) -- (-2,1.5);
        \draw (-1,1) -- (0,2.25);
        \draw (-1,1)node[wnode]{};
    
        \draw (1,1) -- (0,2.25) -- (1.5,1);
        \draw (1,1)node[wnode]{};
        \draw (1.5,1)node[wnode]{};
    
        \draw (0,2.25)node[bnode]{}node[above,nodelabel]{$v'$};

        \end{drawing}
        \caption{$G'$}
        \label{fig}
    \end{subfigure}
    \centering
    \caption{an illustration of bicontracton and bisplitting}
    \label{fig: bicontraction and bisplitting}
\end{figure}

Observe that $G'$ is obtained from $G$ by bicontraction if and only if $G$ may be obtained from $G'$ by bisplitting. Likewise, $G'$ is obtained from $G$ by \bicontract{} if and only if $G$ may be obtained from $G'$ by \bisplit{}. The following lemma pertaining to bicontraction/bisplitting is easily proved, and may also be deduced from \cite[Propositions 2.1 and 2.3]{clm05}.  

\begin{lem} {\sc [Bicontraction preserves Matching Covered property]}\\
    Let $G$ be a graph that has a vertex of degree two, say $v$, that has two distinct neighbours each of which has degree two or more, and let $G':=G/v$. Then $G$ is \mc{} if and only if $G'$ is \mc{}.\qed
    \label{lem : bicontraction preserves mc}
\end{lem}

\begin{figure}[htb]
    \centering  
    \begin{subfigure}{0.45\linewidth}
        \centering
        \begin{drawing}{1.2}
    
        \draw (-2,1.5) -- (2,1.5) -- (2,0.5) -- (-2,0.5) -- (-2,1.5);
        \draw (-1,1) -- (-1,2) -- (-1.5,1);
        \draw (-1,2) -- (0,2.5);
        \draw (-1,1)node[wnode]{};
        \draw (-1.5,1)node[wnode]{};
        \draw (-1,2)node[bnode]{}node[above left,nodelabel]{$v_1$};
    
        \draw (1,1) -- (1,2) -- (1.5,1);
        \draw (0.5,1)--(1,2) -- (0,2.5);
        
        \draw (0.5,1)node[wnode]{};
        \draw (1,1)node[wnode]{};
        \draw (1.5,1)node[wnode]{};
        \draw (1,2)node[bnode]{}node[above right,nodelabel]{$v_2$};
    
        \draw (0,2.5)node[wnode]{}node[above,nodelabel]{$v$};

        \end{drawing}
        \caption{$G$}
        \label{fig:my_label}
    \end{subfigure}
    \begin{subfigure}{0.45\linewidth}
        \centering
        \begin{drawing}{1.2}

        \draw (-2,1.5) -- (2,1.5) -- (2,0.5) -- (-2,0.5) -- (-2,1.5);
        \draw (-1,1) -- (0,2.25) -- (-1.5,1);
        \draw (-1,1)node[wnode]{};
        \draw (-1.5,1)node[wnode]{};
    
        \draw (1,1) -- (0,2.25) -- (1.5,1);
        \draw (0.5,1) -- (0,2.25);
        
        \draw (0.5,1)node[wnode]{};
        \draw (1,1)node[wnode]{};
        \draw (1.5,1)node[wnode]{};
    
        \draw (0,2.25)node[bnode]{}node[above,nodelabel]{$v'$};

        \end{drawing}
        \caption{$G'$}
        \label{fig}
    \end{subfigure}
    \centering
    \caption{an illustration of \bicontract{} and \bisplit{}}
    \label{fig: restricted bicontraction and restricted bisplitting}
\end{figure}


However, bicontraction may not preserve minimality. For instance, the graph shown in Figure~\ref{fig: ex evm} belongs to $\c{H}$, but any application of the bicontraction operation results in a graph that has a removable edge. On the other hand, the \bicontract{} operation preserves minimality as well as \eem{}ity, as proved below.

\begin{thm} {\sc [Restricted Bicontraction preserves $2$-edge Extremality]}\\
    Let $G$ be a bipartite graph that has a vertex of degree two, say $v$, each of whose neighbours has degree three or more, and let $G':=G/v$. Then the following statements hold:
    \begin{enumerate}
        \item $G$ is \mc{} if and only if $G'$ is \mc{};
        \item furthermore, $G$ is minimal if and only if $G'$ is minimal; and
        \item finally, $G$ belongs to $\c{H}_0$ if and only if $G'$ belongs to $\c{H}_0$.
    \end{enumerate}
    \label{thm : bicontraction preserves extremality}
\end{thm}
\begin{proof}
    We adopt notation from the definition of \bisplit{}, as shown in Figure~\ref{fig: restricted bicontraction and restricted bisplitting}. Observe that \i{(i)} follows immediately from Lemma~\ref{lem : bicontraction preserves mc}. Henceforth, we assume that $G$ is \mc{}, or equivalently, that $G'$ is \mc{}. We will prove the other two statements one-by-one.
    \begin{stat}
        $G$ is minimal if and only if $G'$ is minimal.
    \end{stat}
    \begin{proof}
        First suppose that $G'$ is not minimal, and let $e$ be a removable edge. Note that $d_{G-e}(v_i)\geqslant d_{G}(v_i)-1\geqslant2$ since the operation is a \bicontract{}. Consequently, $G'-e$ is obtained from $G-e$ by bicontracting $v$. Thus, by Lemma~\ref{lem : bicontraction preserves mc}, $G-e$ is \mc{}. In other words, $G$ is not minimal.

        Conversely, suppose that $G$ is not minimal, and let $e$ be a removable edge. Note that $e\notin \partial(v)$ since $d(v)=2$. Now, observe that $G'-e$ is obtained from $G-e$ by bicontracting~$v$. Thus, by Lemma~\ref{lem : bicontraction preserves mc}, $G'-e$ is \mc{}. Consequently, $G'$ is not minimal.
    \end{proof}

    Henceforth, we assume that $G$ is a \mbmcg{}, or equivalently, that $G'$ is a \mbmcg{}. Recall that $G$ belongs to $\c{H}_0$ if $|E_2(G)|=m-n+2$. 
    \begin{stat}
        $G$ belongs to $\c{H}_0$ if and only if $G'$ belongs to $\c{H}_0$. 
    \end{stat}
    \begin{proof}
        Let $n':=|V(G')|$ and $m':=|E(G')|$. Note that $m'=m-2$ and $n'=n-2$. Thus, $m'-n'=m-n$. By definition of \bicontract{} operation, $d_G(v_1)\geqslant 3$ and $d_G(v_2)\geqslant 3$; subsequently $d_{G'}(v')\geqslant3$. These observations imply that $|E_2(G')|=|E_2(G)|$. 
    \end{proof}

    This completes the proof of Theorem~\ref{thm : bicontraction preserves extremality}
\end{proof}

Recall the definition of \retract{} from Subsection~\ref{subsec : h0 and h1}. The following is an immediate consequence of the above theorem.

\begin{cor}
    A graph $G$ belongs to $\c{H}_0$ if and only if its \retract{} $\widehat{G}$ belongs to $\c{H}_0$.\qed
    \label{cor : retract in h0}
\end{cor}

We are now ready to prove our characterization of the extremal class $\c{H}_0$ in terms of the extremal class $\c{H}_2$, as restated below.

\begin{restate}{\ref{thm : eem}}{\sc[Characterization of $\c{H}_0$]}\\
    A graph $G$, distinct from $C_4$, belongs to $\c{H}_0$ if and only if its \retract{} $\widehat{G}$ belongs to $\c{H}_2$. 
\end{restate}
\begin{proof}
    First suppose that $\widehat{G}$ belongs to $\c{H}_2$. Observe that either $\widehat{G}=G$ or $\widehat{G}$ has a vertex of degree at least four. In either case, $\widehat{G}\neq C_4$. By the containment established in Corollary~\ref{cor : h2 subset h0}, we infer that $\widehat{G}\in \c{H}_0$. Now, by Corollary~\ref{cor : retract in h0}, we conclude that $G\in \c{H}_0$.

    Now suppose that $G$ belongs to $\c{H}_0$. We invoke Corollary~\ref{cor : retract in h0} to infer that $\widehat{G}\in \c{H}_0$. Thus, by definition, $|E_2(\widehat{G})|=\widehat{m}-\widehat{n}+2$, where $\widehat{n}:=|V(\widehat{G})|$ and $\widehat{m}:=|E(\widehat{G})|$. By definition of \retract{}, the subgraph of $\widehat{G}$ induced by its degree two vertices has no isolated vertex. Applying the handshaking lemma to this subgraph gives us: $|V_2(\widehat{G})|\leqslant 2|E_2(\widehat{G})|=2(\widehat{m}-\widehat{n}+2)$. However, by the lower bound established in Corollary~\ref{cor : bound evm}, $|V_2(\widehat{G})|\geqslant 2(\widehat{m}-\widehat{n}+2)$. Thus, $|V_2(\widehat{G})|= 2(\widehat{m}-\widehat{n}+2)$; by definition, $\widehat{G}\in \c{H}_2$.
\end{proof}


We now switch our attention to the only remaining extremal class. 

\subsection{Characterization of $\c{H}_1$}
\label{subsec : h1}

We begin by deducing a lower bound on the number of $2$-edges in a \mbmcg{} solely in terms of $n$ (see Lov\'asz and Plummer \cite[Lemma 4.2.4 (b)]{lopl86}), that was stated in Subsection~\ref{subsec : bounds}, from the lower bound of Lov\'asz and Plummer (Theorem~\ref{thm : lower bound on 2-edges}).

\begin{cor}
    In a \mbmcg{}, $|E_2|\geqslant \frac{n+10}{6}$.
    \label{cor : bound een}
\end{cor}
\begin{proof}
    Let $G$ be a \mbmcg{}. By Theorem~\ref{thm : lower bound on 2-edges}: 
    \begin{equation}
        |E_2|\geqslant m-n+2
        \label{eq : 1}
    \end{equation}
    
    Now, using the handshaking lemma, and the fact that $d(v)\geqslant 3$ for each $v\in V_3$, we have: $2m\geqslant 3|V_3|+2|V_2|=3n-|V_2|$. In particular:
    \begin{equation}
        2m\geqslant 3n-|V_2|
        \label{eq : 2}
    \end{equation}
    
    By counting the edges incident with vertices of degree two in two different ways, we have: $2|V_2|=2|E_2|+|E_{23}|=|E_2|+m-|E_{3}|$. Now, using the fact that $|E_3|\geqslant 0$, we have:
    \begin{equation}
        |E_2|+m\geqslant 2|V_2|
        \label{eq : 3}
    \end{equation}
    Now, by scaling and adding the above inequalities as $(\ref{eq : 1}) \times 5 + (\ref{eq : 2})\times 2 + (\ref{eq : 3})\times 1$, we arrive at the desired inequality: $6|E_2|\geqslant n+10$.
\end{proof}

Next, we invite the reader to make a simple observation pertaining to the \bicontract{} operation. 


\begin{lem}
    If a graph $G$ has a vertex of degree two, say $v$, each of whose neighbours has degree at least three, then $|E_3(G)|=|E_3(G/v)|$. Consequently, $|E_3(G)|=|E_3(\widehat{G})|$. \qed
    \label{lem : E_3 is preserved}
\end{lem}

We are now ready to prove our characterization of the extremal class $\c{H}_1$ in terms of the extremal class $\c{H}_4$, as restated below. 

\begin{restate}{\ref{thm : een}}{\sc[Characterization of $\c{H}_1$]}\\
    A graph $G$ belongs to $\c{H}_1$ if and only if its \retract{} $\widehat{G}$ belongs to $\c{H}_4$ and $\Delta(G)=3$. 
\end{restate}
\begin{proof}
    For the reverse direction, let $G$ be a graph such that $\widehat{G}\in \c{H}_4$ and $\Delta(G)=3$. In what follows, we shall demonstrate that all of the inequalities that appear in the proof of Corollary~\ref{cor : bound ee} hold with equality. Firstly, by the characterization of $\c{H}_4$ obtained in Theorem~\ref{thm : ee}, we infer that $|E_3(\widehat{G})|=0$. Thus, by Lemma~\ref{lem : E_3 is preserved}, $|E_3(G)|=|E_3(\widehat{G})|=0$. Secondly, by the containment established in Corollary~\ref{cor : h4 subset h2}, $\widehat{G}\in \c{H}_2$. By Theorem~\ref{thm : eem}, $G\in \c{H}_0$. In particular, $|E_2(G)|=m-n+2$.  Lastly, as $\Delta(G)=3$, clearly $d_G(v)=3$ for each $v\in V_3$. The reader may thus verify that equality holds in all of the inequalities discussed in the proof of Corollary~\ref{cor : bound een}. Hence, $|E_2(G)|=\frac{n+10}{6}$ and $G\in \c{H}_1$.

    For the forward direction, let $G$ be a graph in $\c{H}_1$. Consequently, $G$ satisfies all of the inequalities that appear in the proof of Corollary~\ref{cor : bound een} with equality. Firstly, $|E_2(G)|=m-n+2$; thus $G\in \c{H}_0$. Consequently, by Theorem~\ref{thm : eem}, $\widehat{G}\in \c{H}_2$. Secondly, $|E_3(G)|=0$. Thus, by Lemma~\ref{lem : E_3 is preserved}, $|E_3(\widehat{G})|=0$. By Corollary~\ref{cor : counting prop of embmcg} \i{(iii)}, we deduce that $|V_3(\widehat{G})|=2$. Hence, by Lemma~\ref{lem : base case - main theorem} and Theorem~\ref{thm : embmcg TcupT'} \i{(ii)}, we conclude that $\widehat{G}\in \c{H}_4$. Finally, $d_G(v)=3$ for each $v\in V_3(G)$. Thus, $\Delta(G)\leqslant 3$. Note that if $G$ is a cycle then $n=m=E_2=\frac{n+10}{6}$ which implies that $n=2$; a contradiction. Thus, $G$ is not a cycle; whence $\Delta(G)=3$.
\end{proof}

We end this subsection, with a couple of easy consequences. The first of them is the following proper containment that is already established within the above proof.

\begin{cor}
    $\c{H}_1\subset \c{H}_0$. \qed
    \label{cor : h1 subset h0}
\end{cor}

Note that $\Theta$, shown in Figure~\ref{fig: theta}, belongs to $\c{H}_1$ as well as $\c{H}_2$. Conversely, let $G\in \c{H}_1\cap \c{H}_2$. Since $G\in \c{H}_2$, by Lemma~\ref{lem : basic prop of embmcg} \i{(i)}, $E_2$ is a perfect matching of $G[V_2]$ which implies that there is no vertex of degree two each of whose neighbours has degree three or more. Thus, $\widehat{G}=G$. On the other hand, since $G\in \c{H}_1$, by Theorem~\ref{thm : een}, $G\in \c{H}_4$. Furthermore, $\Delta(G)=3$. Observe that $\Theta$ is the only graph in $\c{H}_4$ with maximum degree three. Thus, $G=\Theta$; this proves the following.

\begin{cor}
    $\c{H}_1\cap \c{H}_2=\{\Theta\}$. \qed
    \label{cor : h1 cap h2 = theta}
\end{cor}

In the final section, we present conjectures that are natural generalizations of Theorems \ref{thm : evm}, \ref{thm : evn} and \ref{thm : ee} to \bkex{k}s --- that we proceed to define.

\section{Generalization to \bkex{k}s}
\label{sec : kex}

For a positive integer~$k$, a connected graph, of order least $2k+2$, is \i{$k$-extendable} if it has a matching of size $k$ and each such matching extends to a perfect matching. This notion was first introduced by Plummer \cite{p80}, and has been studied extensively by various authors since then; we impose the additional technical condition of order at least $2k+2$ since otherwise most of our results admit small counterexamples. Observe that a graph is \ex{k} if and only if its underlying simple graph is \ex{k}.  

Note that \mcg{}s are precisely the \kex{1}s. For the sake of completeness, we also use the term \i{\ex{0}} graphs to refer to (not necessarily connected) matchable graphs. Plummer~\cite{p86} proved the following regarding the connectedness of \kex{k}s.

\begin{thm}
    For $k\geqslant 1$, every \kex{k} is \kC{(k+1)}.
    \label{thm : kex is (k+1) connected}
\end{thm}

This immediately yields the following.

\begin{cor}
    In a \kex{k}, each vertex has degree at least $k+1$. \qed
    \label{cor : deg>=k+1}
\end{cor}

A \kex{k} $G$ is \i{minimal} if $G-e$ is not \ex{k} for each edge $e$. We emphasize that the notion of minimality is tied with the value of $k$. For instance, it is easily verified (using Theorem \ref{thm : k-extendability characterization}) that every \ex{2} bipartite cubic graph is a minimal \ex{2} graph, but is not minimal \ex{1}. 


In light of Corollary \ref{cor : deg>=k+1}, it is natural to ask whether there is a lower bound on the number of degree $k+1$ vertices in a minimal \kex{k}. Henceforth, we restrict our attention to bipartite graphs.

We begin by recalling the well-known Hall's Theorem which states that a bipartite graph $G[A,B]$, where $|A|=|B|$, is \ex{0} if and only if $|N(S)|\geqslant |S|$ for each $S\subseteq A$, where $N(S)$ denotes the neighbourhood of $S$. Using this, Plummer \cite{p86} established the following characterization of \bkex{k}s.

\begin{thm}{\sc[Characterization of \ex{k} Bipartite Graphs]}\\
For a bipartite graph $G[A,B]$ of order at least $2k+2$, where $|A|=|B|$, the following are equivalent:
\begin{enumerate}
    \item $G$ is $k$-extendable, 
    \item for every nonempty subset $S\subseteq A$, either $N(S)=B$ or $|N(S)|\geqslant |S|+k$, and
    \item $G-X-Y$ is matchable for all $X\subseteq A$ and $Y\subseteq B$ such that $|X|=|Y|=k$.
\end{enumerate}
\label{thm : k-extendability characterization}
\end{thm}

Note that statement \i{(ii)} in the above theorem is equivalent to enforcing $|N(S)|\geqslant |S|+k$ for every nonempty subset $S\subset A$ such that $|S|\leqslant |A|-k$; the latter is the more commonly used condition in Plummer's work. We also state an immediate corollary of statement \i{(iii)}. 

\begin{cor}
    If $G[A,B]$ is a \bkex{k}, where $k\geqslant 1$, then $G-a-b$ is \ex{(k-1)} for each pair $a\in A$ and $b\in B$. \qed
    \label{cor : G-u-v is k-1 ex}
\end{cor}

We now proceed to discuss a few bounds established by Lou \cite{l99}, and our conjectures pertinent to these.

\subsection{Lou's bounds and our conjectures}
\label{subsec : conjectures}
In the same spirit as in Section~\ref{sec : embmcg}, with respect to a fixed value of $k$, for a \bkex{k} $G$, we let $V_{k+1}(G)$ denote the set of vertices that have degree precisely $k+1$, and we let $V_{k+2}(G)$ denote the set of vertices that have degree at least $k+2$. We drop $G$ from the notation when the graph is clear from the context. For the sake of brevity, we use $E_{k+1}:=E(G[V_{k+1}])$ and $E_{k+2}:=E(G[V_{k+2}])$. 

Note that a \mbkex{0} is precisely a perfect matching. Thus, $|V_{1}|=n$ and $|E|=\frac{n}{2}$. Henceforth, we shall be interested in bounds for $k\geqslant 1$. Lou \cite{l99} proved the following.

\begin{thm}
    In a \mbkex{k} $G$, the induced subgraph $G[V_{k+2}]$ is a forest. 
    \label{thm : forest}
\end{thm}

Using the above, Lou \cite{l99} deduced the following lower bound on the number of vertices of degree $k+1$ in terms of $n$.

\begin{cor}
    In a \mbkex{k}, $|V_{k+1}|\geqslant \frac{kn+2}{2k+1}$.
    \label{cor : lou evn}
\end{cor}

We now use Lou's Theorem (\ref{thm : forest}) to infer another lower bound in terms of $m$ and $n$. 

\begin{cor}
    For $k\geqslant 1$, in a \mbkex{k}, $|V_{k+1}|\geqslant \frac{1}{k}(m-n+1)$.\label{cor : lou evm}
\end{cor}
\begin{proof}
    Note the $(k+1)|V_{k+1}|=|\partial(V_{k+1})|+2|E_{k+1}|\geqslant |\partial(V_{k+1})|+|E_{k+1}|$. Furthermore, by Theorem~\ref{thm : forest}, $|E_{k+2}|\leqslant |V_{k+2}|-1$. Combining these, $m=|E_{k+2}|+|E_{k+1}|+|\partial(V_{k+1})|\leqslant |V_{k+2}|-1 + (k+1)|V_{k+1}|= n-1+k|V_{k+1}|$. By rearranging, $|V_{k+1}|\geqslant \frac{1}{k}(m-n+1)$.
\end{proof}

Using Theorem~\ref{thm : forest}, Lou also deduced the following upper bound on the size.

\begin{cor}
    In a \mbkex{k}, $|E|\leqslant (k+1)n-1$.
\end{cor}

It is worth noting that Lou did not provide examples that satisfy any of the aforementioned bounds tightly. We are also unable to construct such examples. Thus, we conjecture stronger bounds that are generalizations of the corresponding bounds that appear in Table~\ref{tab: extremality notions}. 

\begin{conj}{\sc[Main Conjecture]}\\
    In a \mbkex{k}, $|V_{k+1}|\geqslant \frac{2}{2k-1}(m-n+2k)$.
    \label{conj : evm}
\end{conj}

As we shall see soon, one may prove the following two conjectures assuming the above conjecture.
\begin{conj}
    In a \mbkex{k}, where $k\geqslant 1$, $|V_{k+1}|\geqslant \frac{n}{2}+2$.
    \label{conj : evn}
\end{conj}

For our final conjecture, we are able to construct small counterexamples, and thus impose a lower bound on the order. This is reminiscent of Corollary~\ref{cor : bound ee} wherein $C_4$ appears as the only counterexample.

\begin{conj}
    There exists an integer $N_k\leqslant 4k^2+2k$ such that every \mbkex{k}, on $N_k$ or more vertices, satisfies $|E|\leqslant\frac{(2k+1)(n-2k)}{2}$.
    \label{conj : ee}
\end{conj}

As evidence for the above conjecture, we mention the following result due to Fabres, Kothari and Carvalho \cite{fkc21} whose bound coincides exactly with our conjecture.

\begin{thm}
    Every \mbkex{2}, on twelve or more vertices, satisfies $|E|\leqslant\frac{5n-20}{2}$.
\end{thm}

For each of the above three conjectures, we shall provide tight examples that attain the corresponding bounds, in Subsection \ref{subsec : tight examples}. Our examples may be viewed as straightforward generalizations of the characterizations of the corresponding extremal classes that appear in Theorems~\ref{thm : evm},~\ref{thm : evn} and~\ref{thm : ee}. Below, we prove that our Main Conjecture (\ref{conj : evm}) implies the other two.


\begin{thm}
    Conjecture~\ref{conj : evm} implies Conjecture~\ref{conj : evn}.
\end{thm}
\begin{proof}
    Using the fact that $n=|V_{k+1}|+|V_{k+2}|$, and the handshaking lemma: 
    \[2m=\sum_{v\in V_{k+2}}d(v)+(k+1)|V_{k+1}|\geqslant (k+2)|V_{k+2}|+(k+1)|V_{k+1}|=(k+2)n-|V_{k+1}|\]
    By substituting the above in Conjecture~\ref{conj : evm}, we infer that $(2k-1)|V_{k+1}|\geqslant 2(m-n+2k)\geqslant kn-|V_{k+1}|+4k$. By rearranging, $|V_{k+1}|\geqslant \frac{n}{2}+2$.    
\end{proof}

Note that Conjecture~\ref{conj : ee} already holds when $k=1$ due to Corollary~\ref{cor : bound ee}. Now, we prove Conjecture~\ref{conj : ee} assuming the Main Conjecture (\ref{conj : evm}) for $k\geqslant 2$.

\begin{thm}
    Conjecture~\ref{conj : evm} implies Conjecture~\ref{conj : ee}.
    \label{thm : conj evm implies conj ee}
\end{thm}
\begin{proof}
    Let $G[A,B]$ denote a \mbkex{k}, where $k\geqslant 2$, that satisfies Conjecture~\ref{conj : evm}. Using the fact that $n=|V_{k+1}|+|V_{k+2}|$, we deduce that $2(m-n+2k)\leqslant (2k-1)|V_{k+1}|=(2k-1)(n-|V_{k+2}|)$. If $|V_{k+2}|\geqslant2k$ then $2(m-n+2k)\leqslant (2k-1)(n-2k)$, and by rearranging, we arrive at the desired conclusion: $m\leqslant \frac{2k+1}{2}(n-2k)$. It remains to deal with the case: $|V_{k+2}|<2k$. Before that, we make a few easy observations.
     
    Note that $m=|E_{k+1}|+|\partial(V_{k+1})|+|E_{k+2}|$ and also $|\partial(V_{k+1})|=(k+1)|V_{k+1}|-2|E_{k+1}|$. We deduce that $m=(k+1)|V_{k+1}|-|E_{k+1}|+|E_{k+2}|$. Now, by Theorem~\ref{thm : forest}, $|E_{k+2}|\leqslant|V_{k+2}|-1$. Combining these: 
    
    \begin{equation}
        m\leqslant (k+1)|V_{k+1}|-|E_{k+1}|+|V_{k+2}|-1
        \label{eq : upper bound on m}
    \end{equation}
    
    In what follows, we shall lower bound $|E_{k+1}|$ in order to obtain an upper bound on $m$. For convenience, we use $V^A_{k+2}$ to denote the set $V_{k+2}\cap A$, and likewise for $V_{k+1}^A$ and $V_{k+1}^B$. 
    
    Henceforth, we assume that $|V_{k+2}|<2k$ and adjust notation so that $|V^A_{k+2}|\leqslant k-1$. Consequently, each vertex in $V^B_{k+1}$ has at least two neighbours in $V^A_{k+1}$. As a result, $|E_{k+1}|\geqslant 2|V^B_{k+1}|=2(\frac{n}{2}-|V_{k+2}^B|)\geqslant 2(\frac{n}{2}-|V_{k+2}|)=n-2|V_{k+2}|$. Combining this with Equation~\ref{eq : upper bound on m}, and replacing $|V_{k+1}|$ by $n-|V_{k+2}|$: 

    \[m\leqslant (k+1)|V_{k+1}|-n+2|V_{k+2}|+|V_{k+2}|-1=kn-(k-2)|V_{k+2}|-1\]

    
    Since $k\geqslant 2$, we conclude that $m\leqslant kn-1 <kn$. Finally, observe that $kn\leqslant \frac{(2k+1)(n-2k)}{2}$ if and only if $n\geqslant 2k(2k+1)=4k^2+2k$; this completes the proof of Theorem \ref{thm : conj evm implies conj ee}.
\end{proof}


We now proceed to establish a result on the connectedness of \bkex{k}s that will help us in deducing that the graphs yielded by our constructions, described in Subsection \ref{subsec : tight examples}, are indeed minimal.

\subsection{A result on connectedness of \bkex{k}s}
\label{subsec : essential edge connectivity}


As discussed earlier, Plummer \cite{p86} proved that \kex{k}s are \kC{(k+1)} for $k\geqslant1$; see Theorem~\ref{thm : kex is (k+1) connected}. In this subsection, we prove another interesting property pertaining to the connectedness of \bkex{k}s. We first prove the following technical inequality that surprisingly shows up in our proof of Theorem \ref{thm : essential edge connectivity}. 

\begin{lem}
    Let $p$ and $q$ be nonnegative real numbers where $p<q$, let $D:=[p,q]$, and let $f : D^2\mapsto \m{R}$ be the function $f(x,y):=y(p+q-x)+x(p+q-y)$. Then $f(x,y)\geqslant 2pq$ for each $(x,y)\in D^2$.
    \label{lem : algebra}
\end{lem}
\begin{proof}
    Let $(x,y)\in D^2$. We begin by observing a couple of symmetries: $f(y,x)=f(x,y)=f(p+q-x,p+q-y)$. Consequently, we may adjust notation so that $y\leqslant \frac{p+q}{2}$. Note that $f(x_1,y)-f(x_2,y)=(x_1-x_2)(p+q-2y)$. Thus, $f(x,y)\geqslant f(p,y)\geqslant f(p,p)=2pq$.
\end{proof}

A graph of order four or more is \i{essentially $r$-edge-connected} if each nontrivial cut has size at least $r$.

\begin{thm}
    Every \bkex{k} is essentially $2k$-edge-connected.
\label{thm : essential edge connectivity}
\end{thm}
\begin{proof}
    Let $\partial(W)$ be a nontrivial cut in a \bkex{k} $G[A,B]$; clearly, we may assume $G$ to be simple. We let $W_A:=W\cap A$, and define $W_B,\overline{W}_A$ and $\overline{W}_B$ analogously, and adjust notation so that $W_A$ is a smallest set among the four. We now consider three cases based on $|W_A|$, and argue that $|\partial(W)|\geqslant 2k$ in each case.

    First, suppose that $W_A=\emptyset$. By Corollary \ref{cor : deg>=k+1}, $|\partial(W)|= \sum_{b\in W_B}d(b)\geqslant (k+1)|W_B|$. Now, $|W_B|=|W|\geqslant2$ as $\partial(W)$ is nontrivial. Thus, $|\partial(W)|\geqslant2k+2\geqslant 2k$.

    Next, suppose that $|W_A|\geqslant k$; thus, $|\overline{W}_A|\geqslant k$. Consequently, $|W_A|=|A|-|\overline{W}_A|\leqslant |A|-k$. Thus, by Theorem~\ref{thm : k-extendability characterization}, $|N(W_A)|\geqslant |W_A|+k$. An analogous argument proves that $|N(W_B)|\geqslant |W_B|+k$. Combining  all of these, 
    \[|\partial(W)|\geqslant |N(W_A)- W_B|+|N(W_B)- W_A|\geqslant|N(W_A)|-|W_A|+|N(W_B)|-|W_B|\geqslant 2k\]

    In order to deal with the remaining case, we let $x:=|W_A|$ and $y:=|W_B|$, and make some observations. Each vertex in $W_B$ has at most $x$ edges going to $W_A$. Consequently, by Corollary~\ref{cor : deg>=k+1}, each vertex in $W_B$ has at least $(k+1-x)$ edges going to $\overline{W}_A$. Thus, $|\partial(W)|\geqslant y(k+1-x)$. Using analogous arguments, $|\partial(W)|\geqslant x(k+1-y)$. Since these inequalities are referring to disjoint sets of edges, $|\partial(W)|\geqslant y(k+1-x)+x(k+1-y)$. 
    
    Finally, suppose that $1\leqslant x\leqslant k-1$. If $y\geqslant k$, then $|\partial(W)|\geqslant y(k+1-x)\geqslant 2k$. Otherwise $1\leqslant y\leqslant k-1$. Since $x,y\in [1,k]$, we invoke Lemma~\ref{lem : algebra} to conclude: \[|\partial (W)|\geqslant y(k+1-x)+x(k+1-y) \geqslant 2k\]
    This completes the proof of Theorem \ref{thm : essential edge connectivity}.
\end{proof}

Following the terminology of \cite{fkc21}, an edge $e$ in a $k$-extendable graph $G$ is \i{\removable{}} if $G-e$ is also $k$-extendable. Note that, only for $k=1$, superfluous edges are precisely the removable edges. However, for general $k$, we prefer to use the term \removable{} since the notion of removability appears in the literature extensively; see Lucchesi and Murty \cite{lumu24}. 

Thus, to rephrase, a \kex{k} is minimal if and only if it is free of \removable{} edges. Finally, we record the following easy consequence of Corollary~\ref{cor : deg>=k+1} and Theorem~\ref{thm : essential edge connectivity}. 

\begin{cor}
An edge $e$ of a \bkex{k} is not \removable{} if either of the following holds:
\begin{enumerate}
    \item either an end of $e$ has degree $k+1$,
    \item or $e$ belongs to a nontrivial $2k$-cut. \qed
\end{enumerate}
\label{cor : superfluous edge}
\end{cor}

We shall find the above useful in establishing minimality of the \bkex{k}s obtained by our constructions. 

\subsection{Constructing tight examples for our conjectures}
\label{subsec : tight examples}

We begin by generalizing \htree{}s. A tree is an \i{\ktree{r}} if all of its non-leaves have degree at least $r$, and it is a \i{regular \ktree{r}} if all of its non-leaves have degree exactly $r$. As per this, \htree{}s are precisely the $3$-trees, whereas cubic \htree{}s are precisely the regular $3$-trees. Next, we generalize the \isojoin{} operation.

\begin{defn}{\sc[Isomorphic $k$-leaf Matching]}\\
Let $T$ denote a nontrivial tree, and $L$ its set of leaves. For a positive integer $k$, let $H$ denote the (bipartite) graph obtained from the disjoint union of $k$ copies of $T$, say $T_1,T_2,\dots T_k$, by identifying, for each member of $L$,  all of its $k$ copies into a single vertex; we let $L(H):=L$. Now, let $H'$ denote a copy of $H$. The (bipartite) graph $G$ obtained from $H\cup H'$ by adding a matching, each of whose edges joins a member of $L(H)$ with the corresponding member of $L(H')$ as per some fixed isomorphism between $H$ and $H'$, is said to be obtained from $T$ by \i{\kisojoin{k}}; furthermore, we let $L(G):=L(H)\cup L(H')$.
\label{defn : kisojoin}
\end{defn} 

\begin{figure}[htb]
    \centering
    \begin{subfigure}{0.44\linewidth}
        \centering
        \begin{drawing}{1.2}
            \foreach \x in {-1,1}{
                \draw (3*\x,0)--(1.5*\x,2)--(1.5*\x,0);
                \draw (1.5*\x,2)--(0,2)--(0,0);
            }
            \foreach \x in {-1,1}{
                \draw (3*\x,0)node[bnode]{}(1.5*\x,2)node[wnode]{}(1.5*\x,0)node[bnode]{};
                \draw (1.5*\x,2)node[wnode]{}(0,2)node[bnode]{}(0,0)node[wnode]{};
            }
            \foreach \x in {-3,-1.5,...,3}{
                \draw[color=green] (\x,0) circle (3.5pt);
            }
            \foreach \x in {-1.5,0,1.5}{
                \draw[color=red] (\x,2) circle (3.5pt);
            }

        \end{drawing}
        \caption{tree $T$}
        \label{subfig : k-tree}
    \end{subfigure}
    \begin{subfigure}{0.55\linewidth}
        \centering
        \begin{drawing}{1.2}
            \let\edgecolor\cyanedge
            \foreach \r in {1,2}{
                \foreach \x in {-1,1}{
                    \foreach \k in {-1,1}{
                        \draw (\x*3.5,2*\k)--(\r*\x,0.5*\k);
                        \draw (\x*1.5,2*\k)--(-\r*\x,0.5*\k);
                    }
                    \draw (\x*\r,0.5)edge(\x*\r,-0.5);
                }
            }
            \foreach \x in {1,-1}{
                \foreach \k in {1,-1}{
                    \draw (\x*3.5,2*\k)--(\x*1.5,2*\k);
                    \draw (\x*2.5,2*\k)--(0,0.5*\k);
                }            
            }

            \foreach \x in {1,-1}{
                \draw (\x*3.5,2)node[wnode]{};
                \draw (\x*2.5,2)node[bnode]{};
                \draw (\x*1.5,2)node[wnode]{}; 
                \draw (\x*2,0.5)node[bnode]{};
                \draw (\x*1,0.5)node[bnode]{};
            }
            \foreach \x in {1,-1}{
                \draw (\x*3.5,-2)node[bnode]{};
                \draw (\x*2.5,-2)node[wnode]{};
                \draw (\x*1.5,-2)node[bnode]{};
                \draw (\x*2,-0.5)node[wnode]{};
                \draw (\x*1,-0.5)node[wnode]{};
            }
            \draw (0,0.5)node[wnode]{}edge(0,-0.5)node[bnode]{};
            \foreach \x in {1,-1}{
                \foreach \k in {-1,1}{ 
                    \draw[color=red] (\x*3.5,2*\k) circle (3.5pt);
                    \draw[color=red] (\x*2.5,2*\k) circle (3.5pt);
                    \draw[color=red] (\x*1.5,2*\k) circle (3.5pt); 
                    \draw[color=green] (\x*2,0.5*\k) circle (3.5pt);
                    \draw[color=green] (\x*1,0.5*\k) circle (3.5pt);
                    \draw[color=green] (\x*0,0.5*\k) circle (3.5pt);
                }
            }
    
        \end{drawing}
        \caption{$G$ obtained from $T$ by \kisojoin{2}}
        \label{subfig : k-isojoin}
    \end{subfigure}
    \caption{an example of \kisojoin{k} operation}
    \label{fig : k-isojoin example}
\end{figure}

For instance, the graph $G$ shown in Figure~\ref{subfig : k-isojoin} is obtained from the tree $T$ shown in Figure \ref{subfig : k-tree} by \kisojoin{2}; the green coloured vertices represent $L(G)$ and $L(T)$, respectively. Figures \ref{fig : J_{k,r}} and \ref{fig : double star} also depict examples of the \kisojoin{k} operation. The reader may verify that the \isojoin{} operation defined in Subsection \ref{subsec : h2 characterization} is precisely the \kisojoin{1} operation. We now invite the reader to observe the following that will be used later, together with Corollary~\ref{cor : superfluous edge}, to prove minimality of our examples.

\begin{prop}
Let $G$ be a graph obtained from a tree $T$ by \kisojoin{k} where $k\geqslant 1$; adopt notation from Definition \ref{defn : kisojoin}. For an edge $e$ of $H\cup H'$, let $e_i$ and $e_i'$ denote the copies of $e$ in $T_i$ and $T_i'$, respectively, for each $i\in \{1,2,\dots,k\}$. Then, the set $\{e_1,e_2,\dots ,e_k,e_1',e_2',\dots, e_k'\}$ is a $2k$-cut of $G$ containing the edge $e$. \qed
    \label{prop : 2k-cut}
\end{prop}





We will prove a generalization (namely, Proposition~\ref{prop : mbkex k-isojoin}) of Proposition~\ref{prop : mbmcg TcupT'} in terms of \ktree{(k+1)}s and the \kisojoin{k} operation. In its proof, we induct on the number of non-leaves of the \ktree{(k+1)}, say $T$, to prove that the constructed graphs are indeed \ex{k}. The base case is when the number of non-leaves is at most two. If $T$ has precisely one non-leaf then $T$ is a star. On the other hand, if $T$ has precisely two non-leaves, we call it a \i{double star} and denote it as $D_{p,q}$ where $p$ and $q$ are the degrees of the non-leaves. The induction step of the proof of Proposition \ref{prop : mbkex k-isojoin} will be handled by Lemma \ref{lem : induction step (k,r)-extension}. However, before that, we prove a couple of lemmas to address the base case. 

We define $J_{p,r}$, where $r\geqslant p\geqslant 1$, as the graph obtained from $K_{1,r}$ by \kisojoin{p}; see Figure~\ref{fig : J_{k,r}}. The graph $J_{0,r}$, where $r\geqslant 0$, is defined to be the disjoint union of $r$ copies of $K_2$. The reader may easily observe that $J_{p,r}$ is matchable for all $r\geqslant p\geqslant 0$. For the graph $J_{p,r}[A,B]$ where $r\geqslant p\geqslant 1$, we adopt notation from Definition~\ref{defn : kisojoin}, and we let $A_r:=L(J_{p,r})\cap A$ and $B_r:=L(J_{p,r})\cap B$; on the other hand, for $J_{0,r}[A,B]$, we let $A_r:=A$ and $B_r:=B$. Note that $|A_r|=|B_r|=r$. We now prove the following stronger property.

\begin{lem}
    The bipartite graph $J_{k,r}[A,B]$, where $r\geqslant k\geqslant 0$, is $\min\{k,r-k\}$-extendable.
    \label{lem : base case star}
\end{lem}
\begin{proof}
    Let $S$ be a nonempty subset of $A$. Note that $|A-A_r|=k$. Observe that if $S$ contains a vertex from each of $A-A_r$ and $A_r$, then $N(S)=B$. Otherwise, either $S\subseteq A-A_r$ or $S\subseteq A_r$. In the former case, $|N(S)|=r\geqslant |S|+r-k$ and in the latter case, $|N(S)|=|S|+k$. Thus, by Theorem \ref{thm : k-extendability characterization}, $J_{k,r}$ is \ex{\min\{k,r-k\}}.
\end{proof}
    

\begin{figure}[H]
    \centering
    \begin{drawing}{1}
        \draw (-3.5,0.5) rectangle ++(1,3);
        \draw (3.5,0.5) rectangle ++(-1,3);
        \draw (-1.5,0) rectangle ++(1,4);
        \draw (1.5,0) rectangle ++(-1,4);
        \foreach \r in {0.5,1.5,2.5,3.5}{
            \foreach \k in {1,2,3}{
                \foreach \x in {1,-1}{
                    \draw (\x*1,\r)--(\x*3,\k);
                }
            }
            \draw (-1,\r)node[bnode]{}--(1,\r)node[wnode]{};
        }

        \foreach \k in {1,2,3}{
            \draw (-3,\k)node[wnode]{};
            \draw (3,\k)node[bnode]{};
        }
        \draw (-3,0.1)node[nodelabel]{$A_u$};
        \draw (-1,-0.4)node[nodelabel]{$B_r$};
        \draw (1,-0.4)node[nodelabel]{$A_r$};
        \draw (3,0.1)node[nodelabel]{$B_v$};
    \end{drawing}
    \caption{$J_{3,4}$ obtained from $K_{1,4}$ by \kisojoin{3}}
    \label{fig : J_{k,r}}
\end{figure}

Now, let $G$ be the graph obtained from a double star $D_{p,q}$, where $p,q\geqslant 2$, by \kisojoin{k}. We provide an alternative viewpoint for constructing $G$. Let $M_1, M_2, M_3$ and $M_4$ be the graphs obtained from the disjoint union of $p-1,k,q-1$ and $k$ copies of $K_2$, respectively, and let $A_i$ and $B_i$ denote fixed color classes of $M_i$ for each $i\in \{1,2,3,4\}$. Now, $G$ may be obtained from $M_1\cup M_2\cup M_3\cup M_4$ by joining, for each $i\in \{1,2,3,4\}$, each vertex of $A_i$ with every vertex of $B_{i+1}$, where arithmetic is modulo four; see Figure~\ref{fig : double star}. The following lemma addresses the double star base case. 

\begin{lem}
    For any positive integer $k$, the bipartite graph $G[A,B]$ obtained from a double star $D_{p,q}$, where $p,q\geqslant k+1$, by \kisojoin{k} is \ex{k}.
    \label{lem : base case double star}
\end{lem}

\begin{figure}[H]
    \centering
    \begin{drawing}{1}
        \draw (-0.8,1.2) rectangle ++(1.6,1.35);
        \draw (-0.8,-1.2) rectangle ++(1.6,-1.35);
        \draw (-1.2,-0.8) rectangle ++(-2.1,1.6);
        \draw (1.2,-0.8) rectangle ++(2.85,1.6);
        \foreach \k in {1.5,2.25}{
            \foreach \p in {-1.5,-2.25,-3}{
                \draw (-0.5,\k)--(\p,0.5);
                \draw (\p,-0.5)--(-0.5,-\k);
            }
        }
        \foreach \k in {1.5,2.25}{
            \foreach \q in {1.5,2.25,...,3.75}{
                \draw (0.5,\k)--(\q,0.5);
                \draw (\q,-0.5)--(0.5,-\k);
            }
        }
        \foreach \k in {1.5,2.25}{
            \draw (-0.5,\k)node[wnode]{}--(0.5,\k)node[bnode]{};
            \draw (-0.5,-\k)node[bnode]{}--(0.5,-\k)node[wnode]{};
        }

        \foreach \p in {-1.5,-2.25,-3}{
            \draw (\p,-0.5)node[wnode]{}--(\p,0.5)node[bnode]{};
        }

        \foreach \q in {1.5,2.25,...,3.75}{
            \draw (\q,-0.5)node[bnode]{}--(\q,0.5)node[wnode]{};
        }
        \draw (-3,-1.2)node[nodelabel]{$M_1$};
        \draw (-1.3,2.25)node[nodelabel]{$M_4$};
        \draw (1.3,-2.25)node[nodelabel]{$M_2$};
        \draw (4,1.2)node[nodelabel]{$M_3$};
    \end{drawing}
    \caption{The graph obtained from $D_{4,5}$ by \kisojoin{2}}
    \label{fig : double star}
\end{figure}

\begin{proof}
    We use the alternative viewpoint that was described in the paragraph preceding the lemma statement, and the notation defined therein. Let $A:=A_1\cup A_2\cup A_3\cup A_4$, and likewise for $B$. Let $S$ be any nonempty subset of $A$. If $S$ meets each $A_i$, where $i\in\{1,2,3,4\}$, then $N(S)=B$. Otherwise, there exists an $i\in\{1,2,3,4\}$ such that $S\cap A_i\neq \emptyset$ but $S\cap A_{i+1}= \emptyset$. Now, note that the graph $H:=G-V(M_{i+1})$ is matchable. As a result, $|N_H(S)|\geqslant |S|$. Furthermore, since $S\cap A_i$ is nonempty, $B_{i+1}\subseteq N(S)$. Thus, $|N_G(S)|= |N_H(S)|+|B_{i+1}|\geqslant |S|+k$. Consequently, by Theorem \ref{thm : k-extendability characterization}, $G$ is \ex{k}.
\end{proof}

Before proving Lemma \ref{lem : induction step (k,r)-extension}, we state and prove an easy consequence of Theorem~\ref{thm : k-extendability characterization}.

\begin{cor}
    If $G[A,B]$ is a \bkex{k}, where $k\geqslant 1$, then $G-e-e'$ is \ex{(k-1)} for any two nonadjacent edges $e$ and $e'$ such that an end of $e$ is adjacent with an end of $e'$.  
    \label{cor : G-uy-vx is k-1 ex}
\end{cor}
\begin{proof}
    Let $e:=ab$ and $e':=a'b'$, where $a,a'\in A$ and $b,b'\in B$, so that $ab'\in E(G)$. We let $H:=G-e-e'$. For any $S\subseteq A$, observe that $|N_H(S)|\geqslant |N_G(S)|-2$ and equality holds only if $a,a'\in S$ and $b,b'\notin N_H(S)$. However, since $ab'\in E(H)$, equality does not hold and $|N_H(S)|\geqslant |N_G(S)|-1$. We intend to show that $H$ satisfies statement \i{(ii)} of Theorem \ref{thm : k-extendability characterization} with $k-1$ playing the role of $k$.  
    
    Now, by applying Theorem \ref{thm : k-extendability characterization} to $G$, either $|N_G(S)|\geqslant |S|+k$ or $N_G(S)=B$. If $|N_G(S)|\geqslant |S|+k$ then $|N_H(S)|\geqslant |S|+k-1$, and we are done. Now, suppose that $N_G(S)=B$ and $|N_G(S)|\leqslant |S|+k-1$; these imply that $|S|\geqslant |A|-k+1$. Since $d_H(b)=d_G(b)-1\geqslant k$, we infer that $b\in N_H(S)$; likewise, $b'\in N_H(S)$. Consequently, $N_H(S)=N_G(S)=B$. Thus, by Theorem \ref{thm : k-extendability characterization}, $H$ is \ex{(k-1)}. 
\end{proof}
Next, we describe an operation that appears in the induction step of the proof of Proposition~\ref{prop : mbkex k-isojoin}. Let $G$ be a simple bipartite graph and $uv$ be an edge such that $d(u)=d(v)=p+1$. Recall the definition of $J_{p,r}$ that appears in the paragraph preceding Lemma \ref{lem : base case star}, and the notation therein; let $A_u:=A-A_r$ and $B_v:=B-B_r$. Now, the (bipartite) graph $G'$ --- constructed from the disjoint union of $H:=G-u-v$ and $J_{p,r}$ by adding two matchings, each of size~$p$, one between $N_G(u)-v$ and $A_u$, and another between $N_G(v)-u$ and $B_v$ --- is said to be obtained \i{from $G$ by replacing the edge $uv$ with $J_{p,r}$}. See Figure~\ref{fig : replacing edge by J_k,r} for an illustration. We now show that this operation preserves $k$-extendability.

\begin{lem}
    The bipartite graph $G'$ obtained from a \bkex{k} $G$ by replacing an edge $uv$, where $d(u)=d(v)=p+1$, with $J_{p,r}$ is also \ex{k}.
    \label{lem : induction step (k,r)-extension}
\end{lem}
\begin{proof}
    
    
     We adopt notation from the paragraph preceding the lemma statement, and we let $y_i,u_i,v_i$ and $x_i$, for each $i\in \{1,2,\dots , p\}$, denote the vertices in the sets $N_G(u)-v,A_u,B_v$ and $N_G(v)-u$, respectively. We let $a_i$ and $b_i$, for each $i\in\{1,2,\dots, r\}$, denote the vertices in $A_r$~and~$B_r$, respectively; see Figure \ref{fig : replacing edge by J_k,r}. We proceed by induction on $k$; the following observation proves the statement for $k=0$, and will also come in handy later. 

    \begin{stat}
        For each \pema{} $M$ of $G$, the restriction of $M$ to $G-u-v$ extends to a \pema{} of $G'$.
        \label{stat : matching extension}
    \end{stat}
    \begin{proof}
        First, suppose that $uv\in M$. Then, let $M'$ be any \pema{} of $J_{p,r}$; see Lemma~\ref{lem : base case star}. Then, $M-uv+M'$ is the desired \pema{} of $G'$.

        Otherwise, $p\geqslant 1$. Adjust notation so that $uy_1,vx_1\in M$. Now, let $M'$ be any \pema{} of $J_{p,r}-u_1-v_1\cong J_{p-1,r}$; see Lemma~\ref{lem : base case star}. Then, $M-uy_1-vx_1+M'+u_1y_1+v_1x_1$ is the desired \pema{} of $G'$. 
    \end{proof}

    We use $J$ to denote the subgraph of $G'$ induced by the vertices of $J_{p,r}$. For disjoint sets of vertices of $G'$, say $X$ and $Y$, we let $\partial(X,Y)$ denote the set of those edges whose one end is in $X$ and the other end is in $Y$. We let $M_u, M_r$ and $M_v$ denote the matchings $\partial(N_G(u)-v,A_u), \partial(A_r,B_r)$ and $\partial(B_v,N_G(v)-u)$, respectively, and adjust notation so that both ends of each edge in $M_u\cup M_r\cup M_v$ have the same subscript. 
    
    Now, let $k\geqslant 1$ and $M$ be a matching of size $k$ in $G'$. We consider two cases depending on whether $M$ is a subset of $E(G'-J)$ or not. In each case, we argue that $M$ extends to a \pema{} of $G'$. 
    
    Firstly, suppose that $M\subseteq E(G'-J)$. As $G$ is \ex{k}, there is a \pema{} $M'$ of $G$ containing $M$. Note that the restriction of $M'$ to $G-u-v$ also contains $M$; by~\ref{stat : matching extension}, this extends to a \pema{} of $G'$, and we are done.

    Secondly, suppose that $M\not\subseteq E(G'-J)$. In other words, $M$ contains at least one edge, say $e$, from $E(J)\cup \partial(J)$. Observe that $(M_u,M_v,M_r,\partial(A_u,B_r),\partial(B_v,A_r))$ is a partition of $E(J)\cup \partial(J)$. We now consider two subcases depending on whether $e\in M_u\cup M_r\cup M_v$ or not.
    
    First, suppose that $e\in M_u\cup M_r\cup M_v$. Note that $r\geqslant p\geqslant k$; consequently, $M_u,M_r$ and $M_v$ are disjoint induced matchings of size at least $k$. Using this fact, we choose one edge from each of these sets as follows. For each $F\in\{M_u,M_r,M_v\}$, we pick $e$ if $e\in F$; otherwise, we pick any edge~$e'$ whose both ends are $M$-exposed; let $u_iy_i,v_jx_j$ and $a_\ell b_\ell$ denote these three edges. By Corollary~\ref{cor : G-u-v is k-1 ex}, the graph $H:=G-y_i-x_j$ is \ex{(k-1)}. Also, observe that
    the graph $H':=G'-u_i-y_i-v_j-x_j-a_\ell-b_\ell$ is obtained from $H$ by replacing $uv$ by $J_{p-1,r-1}$. Thus, by the induction hypothesis, $H'$ is \ex{(k-1)}. Consequently, there is a \pema{} $M'$ of $H'$ containing $M-e$. Thus, $M'+u_iy_i+v_jx_j+a_\ell b_\ell$ is a \pema{} of $G'$ containing $M$. 

    \begin{figure}[H]
    \centering  
    \begin{subfigure}{0.24\linewidth}
        \centering
        \begin{drawing}{1}
            \foreach \k in {1,2,3}{
                \foreach \x in {1,-1}{
                    \draw ($(\x*1.5,-\k)+(0,4)$)--(\x*0.5,2);
                }
                \draw ($(-1.5,-\k)+(0,4)$)node[bnode]{}node[below,nodelabel]{$y_\k$};
                \draw ($(1.5,-\k)+(0,4)$)node[wnode]{}node[below,nodelabel]{$x_\k$};
            }
            \draw (-0.5,2)node[wnode]{}node[below,nodelabel]{$u$}--(0.5,2)node[bnode]{}node[below,nodelabel]{$v$};

        \end{drawing}
        \caption{$G$}
        \label{fig:my_label}
    \end{subfigure}
    \begin{subfigure}{0.75\linewidth}
        \centering
        \begin{drawing}{1}
            \draw (-3.5,0.2) rectangle ++(1,3.3);
            \draw (3.5,0.2) rectangle ++(-1,3.3);
            \draw (-5.5,0.2) rectangle ++(1,3.3);
            \draw (5.5,0.2) rectangle ++(-1,3.3);
            \draw (-1.5,-0.3) rectangle ++(1,4.3);
            \draw (1.5,-0.3) rectangle ++(-1,4.3);
            \foreach \r in {1,2,3,4}{
                \foreach \k in {1,2,3}{
                    \foreach \x in {1,-1}{
                        \draw ($(\x*1,-\r)+(0,4.5)$)--(\x*3,\k);
                    }
                }
                \draw ($(-1,-\r)+(0,4.5)$)node[bnode]{}node[below,nodelabel]{$b_\r$}--($(1,-\r)+(0,4.5)$)node[wnode]{}node[below,nodelabel]{$a_\r$};
            }
            
            \foreach \k in {1,2,3}{
                \draw ($(-3,-\k)+(0,4)$)node[wnode]{}node[below,nodelabel]{$u_\k$}--($(-5,-\k)+(0,4)$)node[bnode]{}node[below,nodelabel]{$y_\k$};
                \draw ($(3,-\k)+(0,4)$)node[bnode]{}node[below,nodelabel]{$v_\k$}--($(5,-\k)+(0,4)$)node[wnode]{}node[below,nodelabel]{$x_\k$};
            }
            \draw (-5,-0.2)node[nodelabel]{$N_G(u)-v$};
            \draw (-3,-0.2)node[nodelabel]{$A_u$};
            \draw (-1,-0.7)node[nodelabel]{$B_r$};
            \draw (1,-0.7)node[nodelabel]{$A_r$};
            \draw (3,-0.2)node[nodelabel]{$B_v$};
            \draw (5,-0.2)node[nodelabel]{$N_G(v)-u$};
            \draw (-4,3)node[nodelabel,above]{$M_u$};
            \draw (4,3)node[nodelabel,above]{$M_v$};
            \draw (0,3.5)node[nodelabel,above]{$M_r$};

        \end{drawing}
        \caption{$G'$}
        \label{fig}
    \end{subfigure}
    \centering
    \caption{an illustration of replacing $uv$ by $J_{3,4}$}
    \label{fig : replacing edge by J_k,r}
\end{figure}
    Now, suppose that $e\notin M_u\cup M_r\cup M_v$; consequently,  $e\in \partial(A_u,B_r)\cup \partial(B_v,A_r)$. Adjust notation so that $e:=u_ib_j$. If $a_j$ is matched in $M$ then let $v_\ell\in B_v$ denote its matched neighbour; otherwise, since $p\geqslant k$, let $v_\ell$ denote any unmatched vertex in $B_v$; we shall define a \pema{} $M''$ of $G'$ that contains $M$ in both cases. By Corollary~\ref{cor : G-uy-vx is k-1 ex}, the graph $H:=G-uy_i-vx_\ell$ is \ex{(k-1)}. Also, observe that $H':=G'-u_i-b_j-a_j-v_\ell$ is obtained from $H$ by replacing $uv$ by $J_{p-1,r-1}$. Thus, by the induction hypothesis, $H'$ is \ex{(k-1)}. Consequently, $M-e-a_jv_\ell$ extends to a \pema{} $M'$ of $H'$. Observe that $M'':=M'+u_ib_j+a_jv_\ell$ is the desired \pema{} of $G'$.

    This completes the proof of Lemma~\ref{lem : induction step (k,r)-extension}.
\end{proof}

We are now ready to prove the following generalization of Proposition \ref{prop : mbmcg TcupT'}.

\begin{prop}
    For a positive integer $k$, any graph $G$ --- obtained from a \ktree{(k+1)}~$T$, that is neither $K_2$ nor a star $K_{1,p}$ where $p< 2k$, by \kisojoin{k} --- is a \mbkex{k}.
    \label{prop : mbkex k-isojoin}
\end{prop}
\begin{proof}
    We will induct on the number of non-leaves of $T$. Firstly, if $T$ has exactly one non-leaf, then it is a star and we are done by Lemma~\ref{lem : base case star} as $p-k\geqslant k$. Secondly, if $T$ has exactly two non-leaves, then it is a double star with $p,q\geqslant k+1$ and we are done by Lemma~\ref{lem : base case double star}.  

    Now, suppose that $T$ has at least three non-leaves. Let $u$ be a leaf of the tree obtained from $T$ by deleting all of its leaves. Observe that, in $T$, each neighbour of $u$, except one, is a leaf. Let $T'$ be the tree obtained from $T$ by deleting those leaves that are neighbours of $u$. Let us relate the leaves and non-leaves of $T$ and $T'$ with each other. 
    
    Note that $u$ is a leaf in $T'$. Furthermore, for each vertex $w\in V(T')-u$, we have $N_{T'}(w)=N_{T}(w)$. Thus, the non-leaves of $T'$ are precisely the non-leaves of $T$ minus $u$. Also, each leaf of $T'$, except $u$, is a leaf of $T$. Ergo, $T'$ is also a \ktree{(k+1)} with precisely one non-leaf fewer than $T$; in particular, $T'$ has at least two non-leaves. Hence, by the induction hypothesis, the graph $G'$ obtained from $T'$ by \kisojoin{k} is \ex{k}. 
    
    As $u$ is a leaf in $T'$, let $u$ and $v$ be the two vertices corresponding to $u$ in $G'$. Now, observe that $G$ may be obtained from $G'$ by replacing the edge $uv$ by $J_{k,r}$, where $r:=d_T(u)-1\geqslant k$. Since $G'$ is \ex{k}, by Lemma~\ref{lem : induction step (k,r)-extension}, $G$ is also \ex{k}.   
    
    Next, we prove minimality of $G$. By Corollary~\ref{cor : superfluous edge} \i{(i)}, we only need to inspect those edges of $G$ each of whose ends has degree at least $k+2$; let $e$ denote such an edge. We adopt notation from Definition~\ref{defn : kisojoin}, and observe that $e\in H\cup H'$. By Proposition~\ref{prop : 2k-cut}, the set of edges $\{e_1,e_2,\dots ,e_k,e'_1,e'_2,\dots ,e'_k\}$ is a $2k$-cut (of $G$) that contains $e$. Thus, by Corollary~\ref{cor : superfluous edge}~\i{(ii)}, $e$ is not superfluous. We thus infer that $G$ is minimal.
\end{proof}

Finally, we provide our constructions of the promised tight examples, and use the above proposition to validate them. The reader may compare the statement to Theorem \ref{thm : embmcg TcupT'}.

\begin{thm}
    For a positive integer $k$, any graph $G$ obtained from a \ktree{(k+2)} $T$, that is neither $K_2$ nor a star $K_{1,p}$ where $p< 2k$, by \kisojoin{k} is a \mbkex{k} that satisfies the bound in Conjecture~\ref{conj : evm} with equality. Furthermore:
    \begin{enumerate}
        \item if $T$ is a regular \ktree{(k+2)} then $G$ satisfies the bound in Conjecture~\ref{conj : evn} with equality, whereas
        \item if $T$ is a star then $G$ satisfies the bound in Conjecture~\ref{conj : ee} with equality. 
    \end{enumerate}
    \label{thm : proof for tight examples}
\end{thm}
\begin{proof}
    By Proposition~\ref{prop : mbkex k-isojoin}, $G$ is a \mbkex{k}{}. We adopt all of the notation present in Definition \ref{defn : kisojoin}. We first argue that $G$ satisfies the bound in Conjecture \ref{conj : evm} with equality.  
    
    Since $T$ is a \ktree{(k+2)}, $V_{k+1}(G)=L(G)$. Note that $|E_{k+1}|=\frac{1}{2}|L(G)|=\frac{1}{2}|V_{k+1}|$ and $|\partial(V_{k+1})|=k|V_{k+1}|$. Consequently, $|E_{k+2}|=m-|E_{k+1}|-|\partial(V_{k+1})|=m-\frac{2k+1}{2}|V_{k+1}|$. 
    
    On the other hand, observe that $G[V_{k+2}]=G-V_{k+1}$ is a forest with $2k$ components. So, $|E_{k+2}|=|V_{k+2}|-2k=n-|V_{k+1}|-2k$. 
    
    It follows from the preceding two paragraphs that $m-\frac{2k+1}{2}|V_{k+1}|=n-|V_{k+1}|-2k$. By rearranging, $\frac{2k-1}{2}|V_{k+1}|=m-n+2k$ which is precisely the bound in Conjecture~\ref{conj : evm}. It remains to prove statements \i{(i)} and \i{(ii)}.
    
    First, suppose that $T$ is a regular \ktree{(k+2)}. By double counting, $\sum_{v\in V_{k+2}}d(v)=2|E_{k+2}|+|\partial(V_{k+2})|$. As noted earlier, $|E_{k+2}|=|V_{k+2}|-2k$ and $|\partial(V_{k+2})|=k|V_{k+1}|$. By substituting, $(k+2)|V_{k+2}|=2(|V_{k+2}|-2k)+k|V_{k+1}|$. By rearranging and dividing by $k$, $|V_{k+2}|=|V_{k+1}|-4$. Now, by plugging $|V_{k+2}|=n-|V_{k+1}|$ and rearranging, $|V_{k+1}|=\frac{n}{2}+2$ which is precisely the bound in Conjecture~\ref{conj : evn}.
    
    Finally, suppose that $T$ is a star $K_{1,p}$, where $p\geqslant2k$. Note that $n=2p+2k$. Observe that $m=(2k+1)p=\frac{(2k+1)(n-2k)}{2}$ which is precisely the bound in Conjecture~\ref{conj : ee}. This completes the proof of Theorem \ref{thm : proof for tight examples}.
\end{proof}

We are unable to construct any examples, apart from the ones described in the above theorem statement, that satisfy the bounds in Conjectures~\ref{conj : evm},~\ref{conj : evn} or~\ref{conj : ee} with equality, and are thus tempted to further conjecture that these are the only such examples.

\bibliographystyle{plainurl}
\bibliography{clm}

\end{document}